\documentclass[10pt,a4paper,reqno]{amsart}
\usepackage{amsmath,amsthm,amssymb}
\usepackage{mathrsfs,mathtools,nccmath,enumerate}
\usepackage{tikz-cd}

\usepackage[thinlines]{easytable}
\usepackage{amsfonts}
\usepackage{multicol}
\usepackage[mathscr]{euscript}

\usepackage{graphicx,caption,subcaption,float}
\usepackage{hyperref}

\usepackage[super]{nth}

\newtheorem{theorem}{Theorem}[section]
\newtheorem{lemma}[theorem]{Lemma}
\newtheorem{proposition}[theorem]{Proposition}

\theoremstyle{definition}
\newtheorem{definition}{Definition}[section]

\newtheorem{example}{Example}[section]
\theoremstyle{remark}
\newtheorem{remark}[theorem]{Remark}

\usepackage[export]{adjustbox} 

\begin{document}

		\title[Arc shift move and region arc shift move for twisted knots]{Arc shift move and region arc shift move for twisted knots}
	
	
	\author[T. MAHATO]{Tumpa Mahato}
	\address{Department of Mathematics,\\ Indian Institute of Technology, Ropar, Punjab 140001, India}
	\email{staff.tumpa.mahato@iitrpr.ac.in}
	
	\author[P. MADETI]{Prabhakar Madeti}
	\address{Department of Mathematics,\\ Indian Institute of Technology, Ropar, Punjab 140001, India}
	\email{prabhakar@iitrpr.ac.in}

	\makeatletter
	\@namedef{subjclassname@2020}{%
		\textup{2020} Mathematics Subject Classification}
	\makeatother
	
	\thanks{This work was supported by  Anusandhan National Research Foundation (ANRF), Government of India(Grant no: CRG/2023/004921/343)}

	\begin{abstract}
		In this paper, we study the unknotting operation for twisted knots, called arc shift move. First, we find a family of twisted knots with arc shift number $n$ for any given $n \in \mathbb{N}$. Then we define a new unknotting operation, called the region arc shift move for twisted knots and find  family of twisted knots  whose region arc shift number is less than or equal to $n$ for any given $n \in \mathbb{N}$. Later, we explore bounds for region arc shift number and forbidden number.
	\end{abstract}

	\maketitle
	\section{Introduction}\label{sec:intro}
	
	Twisted knots were introduced by Bourgoin\cite{Bur} as a generalization of virtual knot theory. They can be described by planar diagrams containing classical crossings, virtual crossings, and bars on arcs, together with suitable Reidemeister-type moves. This theory includes both classical and virtual knots as special cases, but it also exhibits new behavior that does not appear in either setting.
	In particular, crossing changes do not suffice to unknot twisted knots, motivating the study of alternative unknotting operations and their associated numerical invariants. To address this problem, Negi and Prabhakar \cite{Negi} extended the arc shift move from virtual knots to twisted knots and defined the arc shift number as the minimum number of arc shift moves required to transform a twisted knot into the unknot.
	
	A natural question is whether there exists a twisted knot with arc shift number $n$ for any given $n \in \mathbb{N}$. In this paper, we answer this question by constructing families of twisted knots where each knot in he family has arc shift number $n(\in \mathbb{N})$. In addition, we define a new diagrammatic move called the region arc shift move for twisted knot and prove that this move is an unknotting operation. Then, we introduce two new numerical invariants: the forbidden number and the region arc shift number. It is important to find strict bounds for any unknotting number as it aids in computing its exact value for any knot. Therefore, we find lower bounds for forbidden number and upper bounds for region arc shift number using its relation with other existing invariants.

	This paper is organized as follows. Section~\ref{sec:pre} reviews basic definitions and results related to twisted knots, arc shift move. It also briefly discusses the polynomial invariant $Q(s,t)$ \cite{Naoko1} that we use to distinguish knots in the twisted knot families in Subsection~\ref{sec:polyinv}. Section~\ref{sec:fam} provides families of twisted knots that realizes any positive integer as its arc shift number. In Section~\ref{sec:ras}, we define an unknotting operation, called the region arc shift move. Section~\ref{sec:rasno} discusses two numerical invariants of twisted knots: the region arc shift number and the forbidden number and explores their bounds.
	
		\begin{figure}[h!]
		\centering
		\subfloat[Classical Reidemeister Moves]{\includegraphics[width=0.75\textwidth]{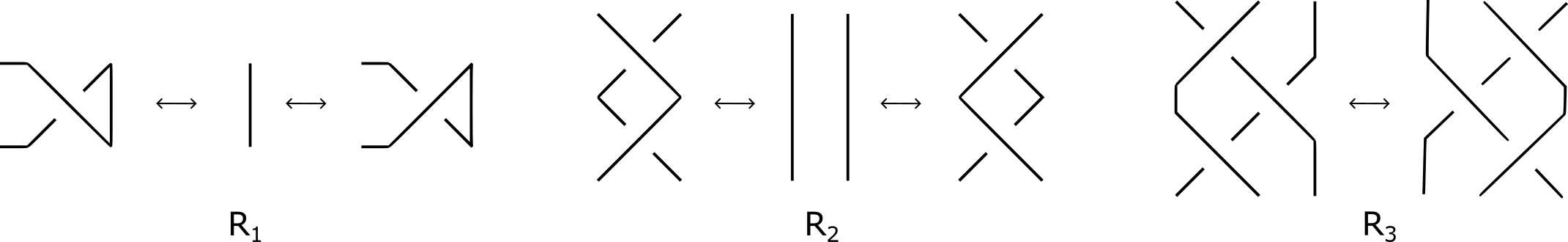}}
		\quad
		
		\subfloat[Virtual Reidemeister Moves]{\includegraphics[width=0.65\textwidth]{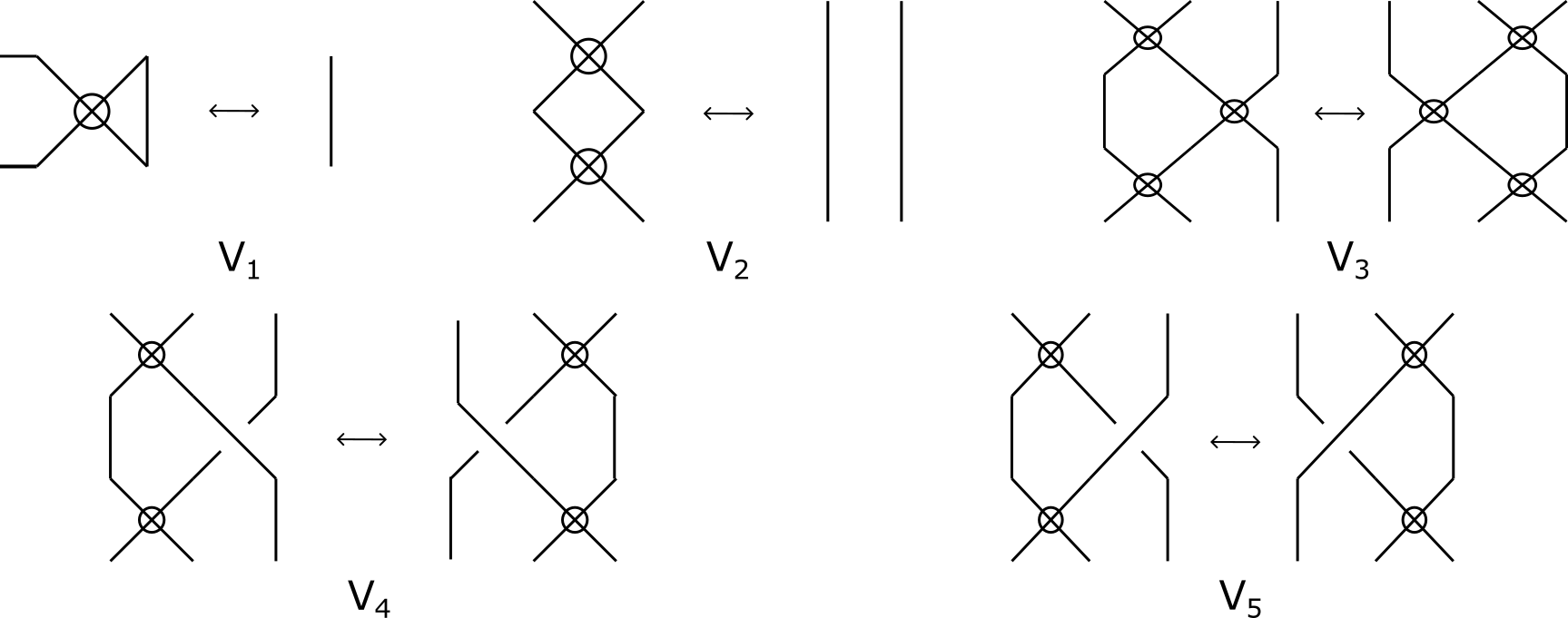}} 
		\newline
		\subfloat[Twisted Moves]{\includegraphics[width=0.75
			\textwidth]{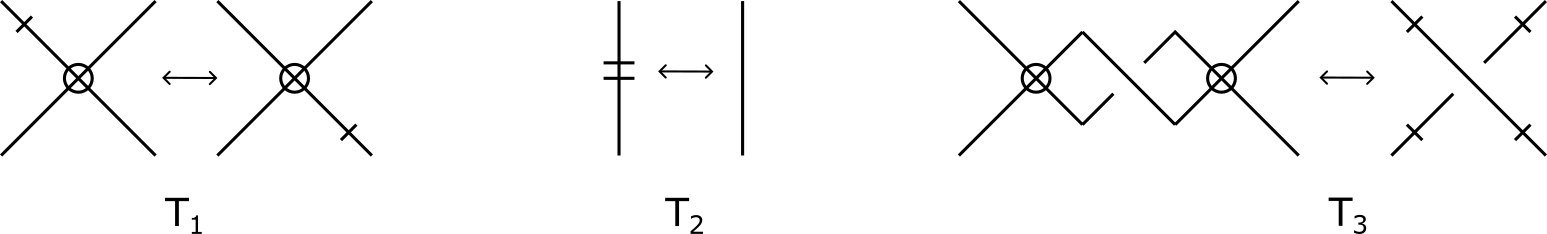}}
		\caption{Extended Reidemeister moves.}
		\label{fig:moves}
	\end{figure}

	\section{Preliminaries}\label{sec:pre}
	Twisted link diagrams are defined as marked generic planar curves, where the
	markings identify the usual classical crossings, virtual crossings, and bars on edges. New Reidemeister moves $T_{1},T_{2},T_{3}$ are included with the existing generalized Reidemeister moves $R_{1}, R_{2}, R_{3}, V_{1}, V_{2}, V_{3}, V_{4}$ (Fig.~\ref{fig:moves}) to define equivalence relation between twisted link diagrams. 
 
	\begin{remark}
		A trivial twisted knot is a classical unknot without a bar or with one bar.
	\end{remark}
	Gauss code and Gauss diagram are also generalized for twisted knot diagrams by recording bars on the arcs between two classical crossings (Fig.~\ref{fig:gausscode}). 
	
	\begin{figure}[h]
		\centering
		\includegraphics[width=0.35\textwidth]{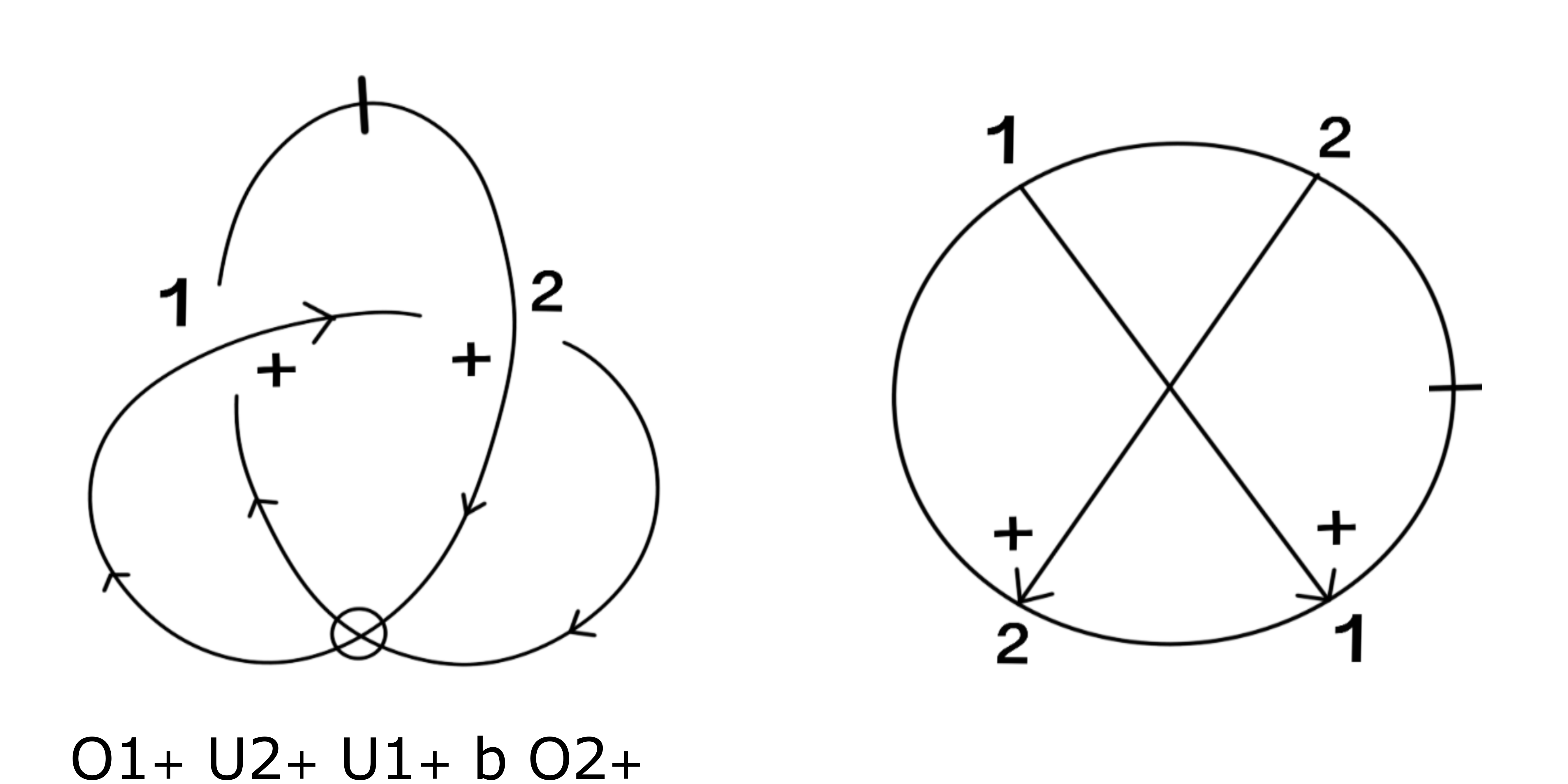}
		\caption{Gauss code and Gauss diagram of twisted knots.}
		\label{fig:gausscode}
	\end{figure}
	Later forbidden moves $F_{3} (\text{or} F_{4})$ and $T_{4}$ (Fig.~\ref{fig:forbidden1} and \ref{fig:forbidden2}) are introduced for twisted knot diagrams and Xue and Deng \cite{Xue} proved the following. 
	\begin{theorem}\label{th:for}\cite{Xue}
		Any Gauss diagram of twisted knot can be changed into Gauss diagram of a trivial
		knot (with a bar) by a sequence of moves of types $R_1, R_2, R_3, T_2, T_3$ and forbidden moves $T_4,
		F_1$ (or $F_2$) and $F_3$ (or $F_4$).
	\end{theorem}	
	\begin{figure}[h!]
	\centering
	\subfloat[Forbidden moves $F_1$ and $F_2$.]{\includegraphics[width=0.6\textwidth]{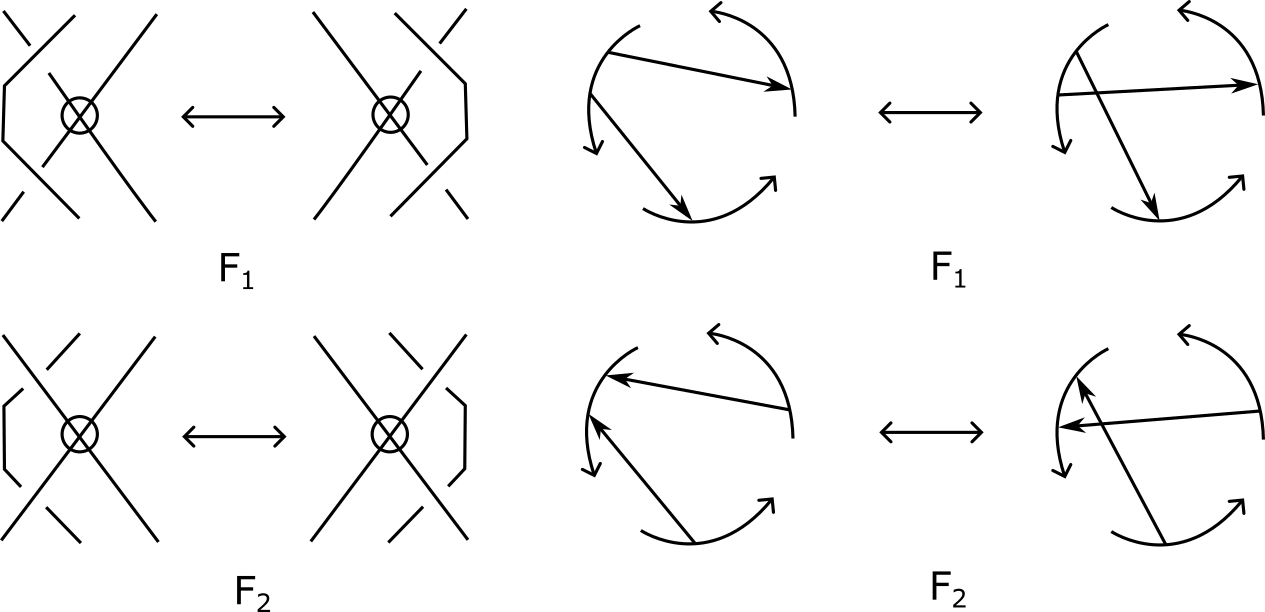}}
	\quad
	\subfloat[Forbidden moves $F_3$ and $F_4$.]{\includegraphics[width=0.6\textwidth]{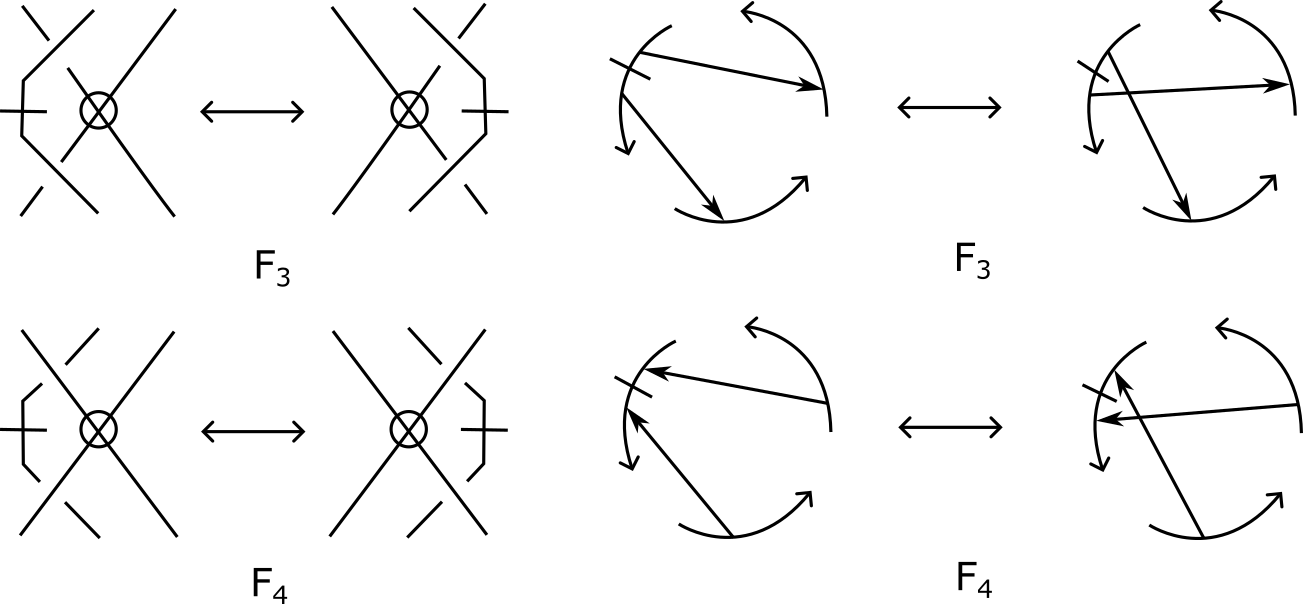}} 
	\caption{Forbidden Moves of twisted knots and the corresponding Gauss diagrams.}
	\label{fig:forbidden1}
\end{figure} 
\begin{figure}[h!]
	\centering
	\includegraphics[width=0.7\textwidth]{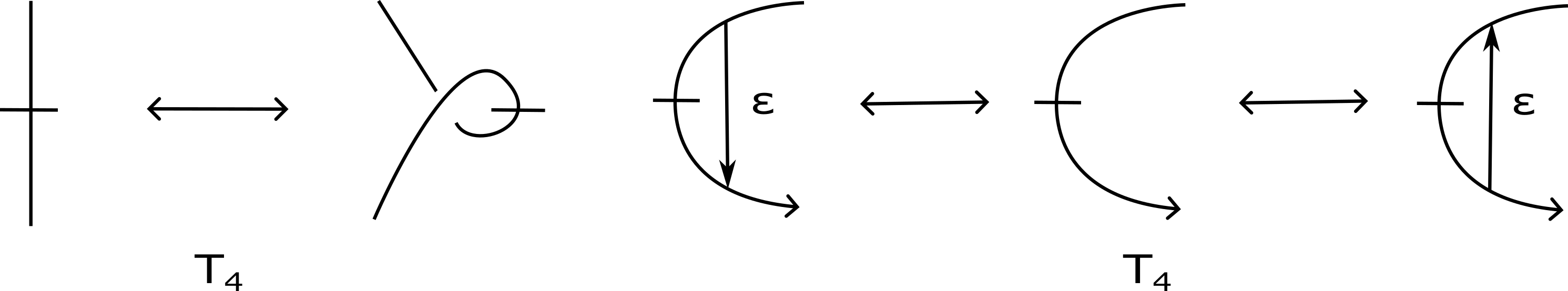}
	\caption{Forbidden move $T_4$.}
	\label{fig:forbidden2}
\end{figure}

	\subsection{Arc shift move and arc shift number of twisted knots}
	Gill, Prabhakar and Vesnin in \cite{Gill2} defined an unknotting operation for virtual knots called {\it arc shift move}, which was later generalized for twisted knots by Negi and Prabhakar \cite{Negi}. In this section, we summerize the definitions of arc shift move, arc shift number and results regarding this.
	\begin{definition}
		An arc in a twisted knot diagram, say $(a, b)$ is defined to be the segment
		passing through exactly one pair of crossings (classical/virtual) $(c_1, c_2)$ with $a$ incident
		to $c_1$ and $b$ incident to $c_2$. Between these crossing points, bars may or may
		not appear. In Fig.~\ref{fig:arc}, arcs $(a, b)$, and $(c, d)$ passes through crossings $(c_{1},c_{2})$.
	\end{definition}
	\begin{figure}
		\centering
		\includegraphics[width=0.16\textwidth]{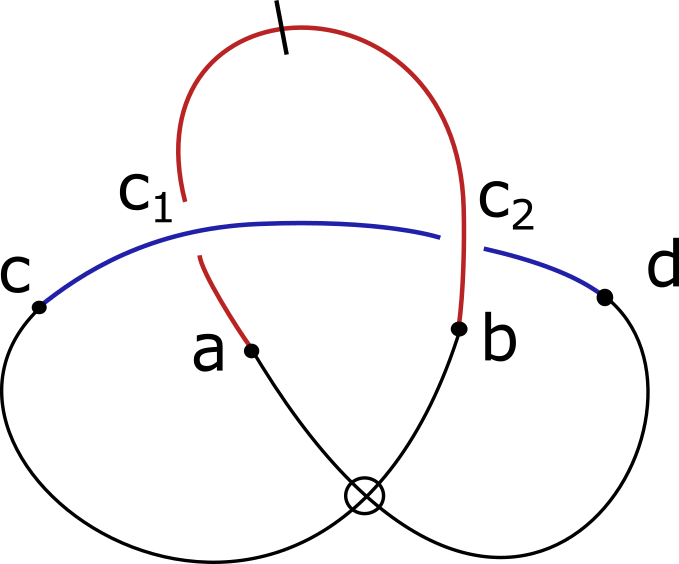}
		\caption{Arc $(a,b)$ and arc $(c,d)$.}
		\label{fig:arc}
	\end{figure}
	\begin{figure}[h!]
		\centering
	\includegraphics[width=0.43\textwidth,height=0.15\textwidth]{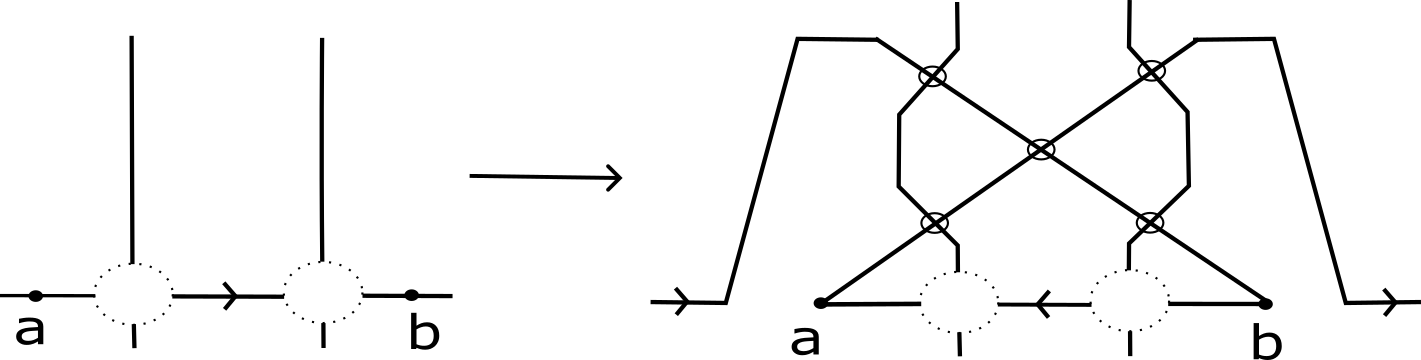}
		\caption{Arc shift move on the arc $(a,b)$.}
		\label{fig:arc move}
	\end{figure} 
		\begin{figure}[h!]
		\centering
		\subfloat{	\includegraphics[width=0.41\textwidth]{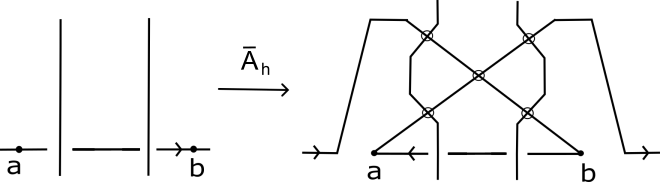}}
		\qquad
		\subfloat{	\includegraphics[width=0.41\textwidth]{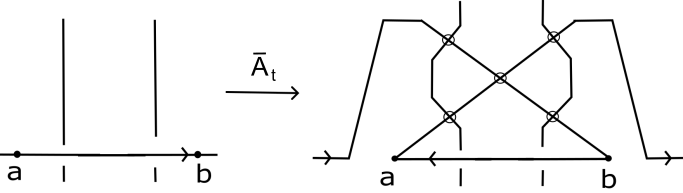}}
		\newline
		\subfloat{	\includegraphics[width=0.41\textwidth]{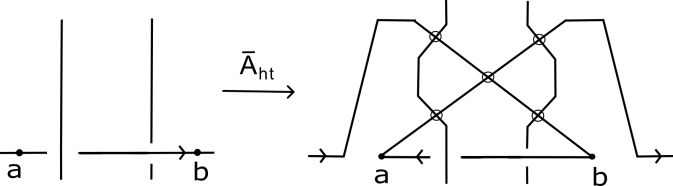}}
		\qquad
		\subfloat{	\includegraphics[width=0.41\textwidth]{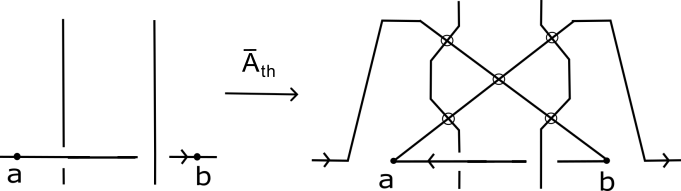}}
		\newline
		\subfloat{	\includegraphics[width=0.41\textwidth]{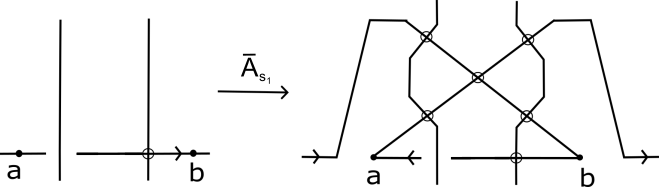}}
		\qquad
		\subfloat{	\includegraphics[width=0.41\textwidth]{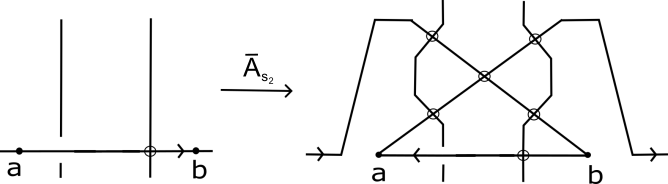}}
		\caption{All Type 1 arc shift moves.}
		\label{fig:type1}
	\end{figure}
	\begin{figure}[h!]
		\centering
		\subfloat{	\includegraphics[width=0.41\textwidth]{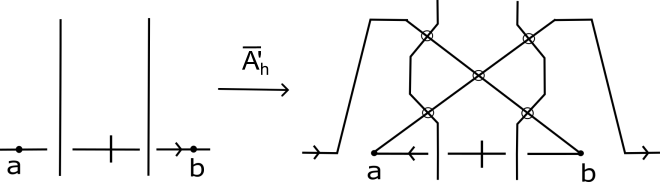}}
		\qquad
		\subfloat{	\includegraphics[width=0.41\textwidth]{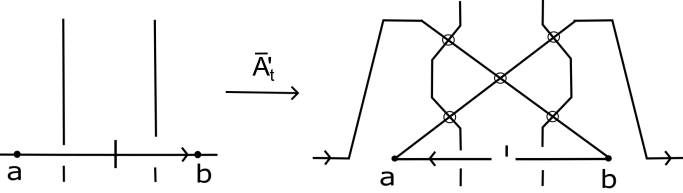}}
		\newline
		\subfloat{	\includegraphics[width=0.41\textwidth]{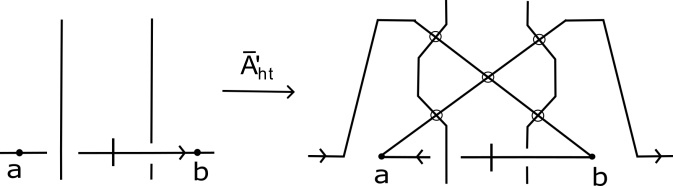}}
		\qquad
		\subfloat{	\includegraphics[width=0.41\textwidth]{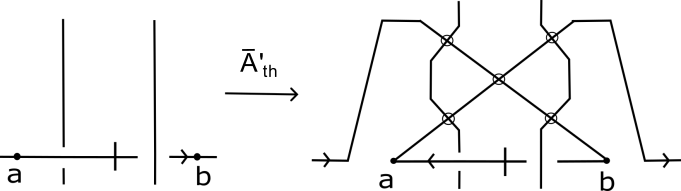}}
		\newline
		\subfloat{	\includegraphics[width=0.41\textwidth]{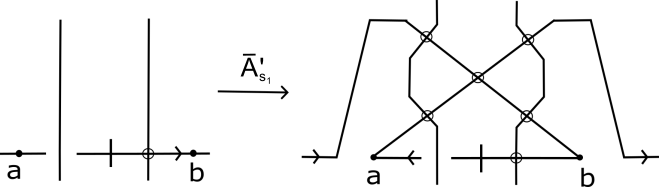}}
		\qquad
		\subfloat{	\includegraphics[width=0.41\textwidth]{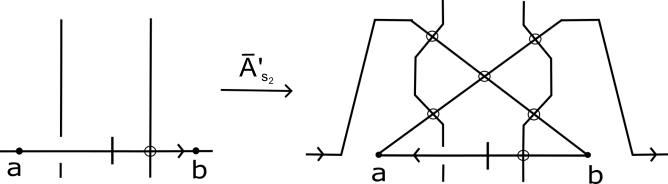}}
		\caption{All Type 2 arc shift moves.}
		\label{fig:type2}
	\end{figure}

	\begin{definition} 
		In a twisted knot diagram $D$, let $(a, b)$ be an arc passing through the pair of crossings $(c_{1},c_{2})$. The dotted circles \big(\includegraphics[height=4ex,valign=m]{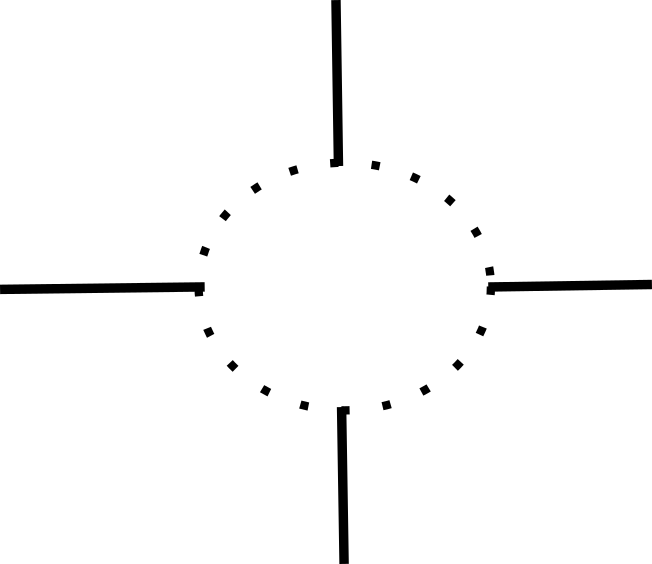}\big) in Fig.~\ref{fig:arc move} represent the fact that the crossings can be classical or virtual.  To apply an arc shift move on the arc
		$(a, b)$, we cut the arc at two points near $a$ and $b$ and identify the loose ends on one side with loose ends on the other side in the way as
		shown in Fig.~\ref{fig:arc move}. We label all the new crossings that arise while applying this move as virtual crossings.
	\end{definition}
	\begin{remark}
		\begin{itemize}
			\item[1.]Depending on the existence of a bar on the arc, there are two types of arc shift move, called Type 1 and Type 2. Also, the choice of $c_1$ and $c_2$ being virtual or classical gives various local configurations of arc shift move. In Fig.~\ref{fig:type2}, all Type 2 moves are shown. For Type 1 arc shift moves, refer to \cite{Gill2}.
			\item[2.] Notice, after applying arc shift move on an arc, the order and signs of the crossings $c_1$ and $c_2$ are altered.
		\end{itemize}
		
	\end{remark}
	\begin{theorem}\cite{Negi}
		Every twisted knot diagram D can be transformed into a trivial
		twisted knot diagram using arc shift moves and extended Reidemeister moves.
	\end{theorem}
	\begin{definition}\cite{Negi}
		For any given twisted knot $K$, its arc
		shift number denoted as $A(K)$ is determined as the minimum value of the set
		$\{A(D)|D \; \text{represents}\; K\}$, where $A(D)$ signifies the minimum number of arc shift
		moves required to transform a diagram D into a trivial twisted knot.
	\end{definition}
	\begin{theorem}\cite{Negi}
		Arc shift number is an invariant for twisted knots.
	\end{theorem}
	\subsection{Odd writhe of a twisted knot}
	For a classical crossing $c$ in a twisted knot diagram $D$, the index of $c$, denoted by $ind(c)$, is defined by the  number of classical
	crossings while traversing along the diagram on one full path that starts at $c$ and returns
	to $c$ \cite{Kau2}. A classical crossing $c$ is said to be an {\it odd crossing} if  $ind(c)$ is odd. Let $Odd(D)$ be the set of all odd crossings in $D$.
	
	\begin{definition}\cite{Negi}
		The odd writhe number $J (D)$ of $D$ is equal to the sum of the signs of all
		the odd crossings in D, i.e.,
		\[J (D) = \sum_{c\in Odd(D)} sgn(c).\]
	\end{definition}
	
	\begin{theorem}\cite{Negi}
		The odd writhe $J(K)$ is an invariant for the twisted knot $K$.
	\end{theorem} 
	\begin{lemma}\cite{Negi}\label{le:odd}
		Odd writhe $J(K)$ forms a lower bound for arc shift number $A(K)$ of twisted knot $K$.
		\[ \Big| \frac{J(K)}
		{2} \Big|
		\leq A(K).\]
		
	\end{lemma}
	
	\subsection{A polynomial invariant based on Affine Indices for the twisted knots}\label{sec:polyinv}
	In this section, we briefly describe a polynomial invariant $Q(s,t)$ of a twisted link \cite{Naoko1} which is defined using the affine indices of the classical crossings of a twisted knot diagram. We use this polynomial to distinguish the knots in each family of twisted knots in Section~\ref{sec:fam}. Moreover, we find lower bound of the forbidden number in Section~\ref{sec:rasno} by observing the changes in $Q(s,t)$ under each forbidden moves $F_1$(or $F_2$), $F_3$(or $F_4$) and $T_4$.
	
	For a twisted link diagram $D$, let $SelfX(D)$ be the set of classical self-crossings of $D$ and $D_c$ be the twisted link diagram obtained from $D$ by smoothing at $c$ for $c \in Self X(D)$. 
	The over-counting component $D_c^O$ and the under-counting component $D_c^U$ for $D$ at $c$ are defined as in Fig.~\ref{fig:component}. 
	
	\begin{figure}[h!]
		\centering
		\subfloat[At a positive crossing]{\includegraphics[width=0.4\textwidth]{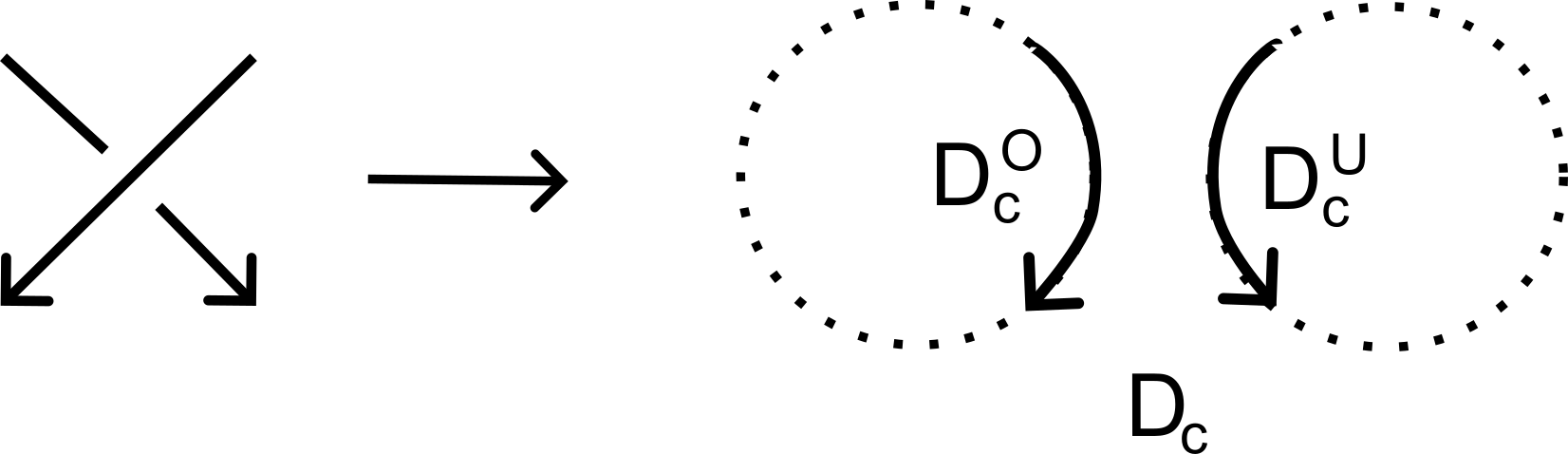}}
		\quad \quad
		\subfloat[At a negative crossing]{\includegraphics[width=0.4\textwidth]{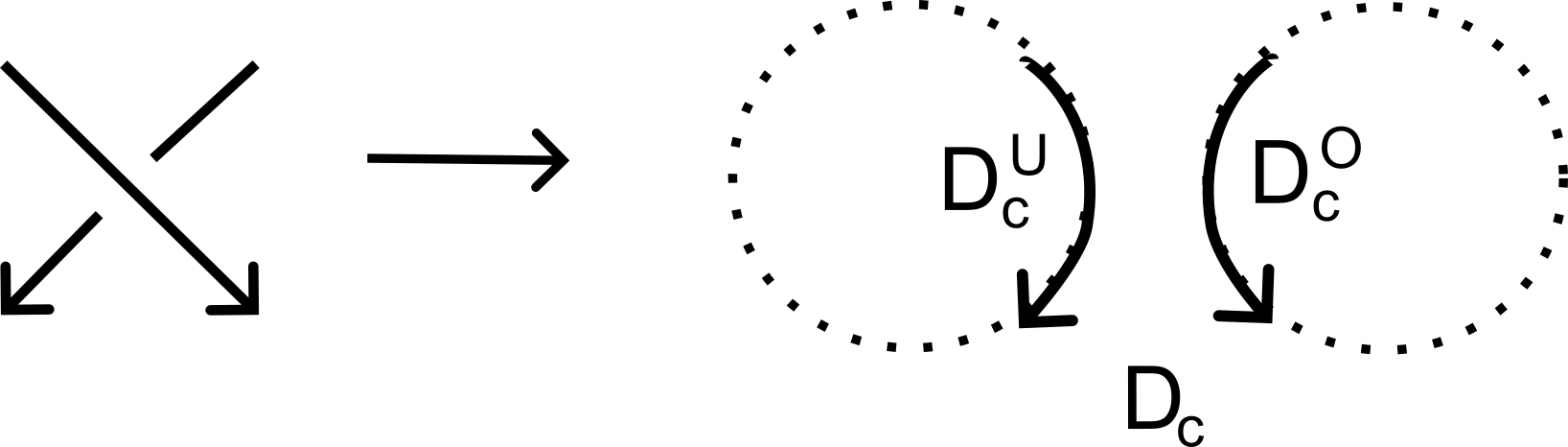}} 
		\caption{Counting components of $D$ at $c$.}
		\label{fig:component}
	\end{figure} 
	
	Let $\gamma_1, \gamma_2, \cdots,\gamma_k$ be the sequence of classical crossings of $D_{c}^{*}  (*\in {O,U
	})$ which appears on the path while transversing the diagram $D_{c}^{*}  (*\in {O,U
	})$ from a point near the crossing $c$, then back to the point. Here if there is a self-crossing of $D_{c}^{*}$, then it appears twice in the sequence $\gamma_1, \gamma_2, \cdots,\gamma_k$.
	Define the $\text{flat sign}$ at $\gamma_i$, denoted by $i(\gamma_i, D_c^*)$, as follows.
	\[i(\gamma_i, D_c^*) = \begin{cases}
		
		\; \; \, sgn(\gamma_i), \quad \text{if} \, \gamma_i \, \text{is an under crossing,}\\
		-sgn(\gamma_i), \quad \text{if} \, \gamma_i \, \text{is an over crossing.}
		
	\end{cases}\]
	Then, the over and under affine index is defined by 
	\[ind^*(c,D) = \sum_i i(\gamma_i, D_c^*) \qquad (*\in {O,U
	})\]
	For $*\in \{O, U\}$, define \[  \bar{\rho}^*(c) = \begin{cases}
		1, \quad  \text{if} \; ind*(c)\; \text{is odd} \\
		0, \quad  \text{if} \; ind*(c)\; \text{is even}
	\end{cases}    \]
	\[  p^*(c) = \begin{cases}
		1, \quad  \text{if the number of bars on} \;  D^*_c \; \text{is odd} \\
		0, \quad  \text{if  the number of bars on}\; D^*_c \;  \text{is even}
	\end{cases}    \]
	\begin{definition}\cite{Naoko1}\label{def:Q}
		For a twisted link diagram $D$, we define a polynomial $Q_D(s,t)$ by 
		\[ Q_D(s,t) = \sum_{c \in Self(D)} sgn(c) (s^{\bar{\rho}^{O}(c)} t^{p^{O}(c)} -1) (s^{\bar{\rho}^{U}(c)} t^{p^{U}(c)} -1)\]
	\end{definition}
	\begin{theorem}\cite{Naoko1}
		The polynomial $Q_D(s,t)$ is an invariant of a twisted link.
	\end{theorem}
	\begin{remark}\cite{Naoko1}\label{rem:over}
		When $D$ is a twisted knot diagram, $\bar{\rho}^{O}(c) = \bar{\rho}^{U}(c)$. Hence, we may compute one of $\bar{\rho}^{O}(c)$ and $\bar{\rho}^{U}(c)$ for each $c\in SelfX(D)$.
		
	\end{remark}
	
	\section{Twisted knots with arc shift number $n$}\label{sec:fam}
	
	In this section, we answer the question that was mentioned in Section~\ref{sec:intro} about finding families of twisted knots with arc shift number $n$ for any given $n \in \mathbb{N}$. Further, we distinguish the knots in the families using the polynomial invariant $Q(s,t)$ \cite{Naoko1}.

	\begin{theorem}\label{th:Kn}
		For each positive integer $n$, there exists a twisted knot $K_n $ such that the arc shift number of each \(K_n\) is exactly equal to \(n\).
	\end{theorem}
	
	\begin{proof}
	To prove this theorem, we construct a sequence of twisted knots $K_1, K_2, \cdots  , K_n$ by using the local diagram in Fig.~\ref{fig:block1}, as a building block . The $i^{th}$ block $B_i$ contains the classical crossings $c_i,d_i$, and $d'_i$. And each $K_i$ is obtained from $K_{i-1}$ by adding one block for all $2 \leq i \leq n$. Now, consider the twisted knot diagram $K_n$ as shown in Fig.~\ref{fig:diagram1}. We prove that $K_n$ has arc shift number $n$. Let $G(K_n)$ be the Gauss diagram for $K_n$ (Fig.~\ref{fig:gauss}).
		\begin{figure}[h]
		\centering 
		\includegraphics[width=0.26\textwidth]{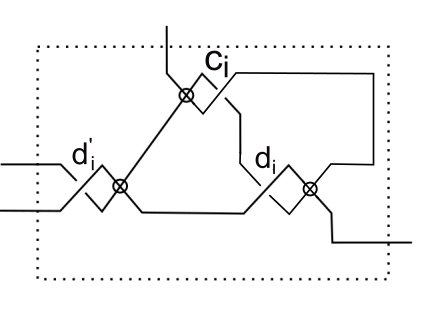}
		\caption{Building block $B_i$.}
		\label{fig:block1}
	\end{figure}
		\begin{figure}[h!]
			\centering 
			\includegraphics[width=0.9\textwidth]{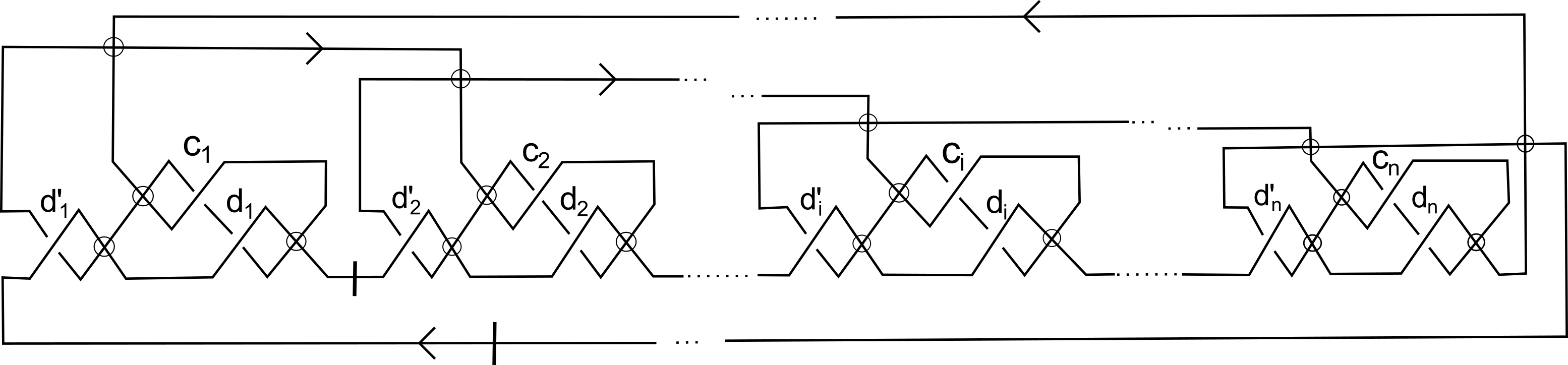}
			\caption{$K_n.$}
			\label{fig:diagram1}
		\end{figure}
		\begin{figure}[h]
			\centering 
			\includegraphics[width=0.28
			\textwidth]{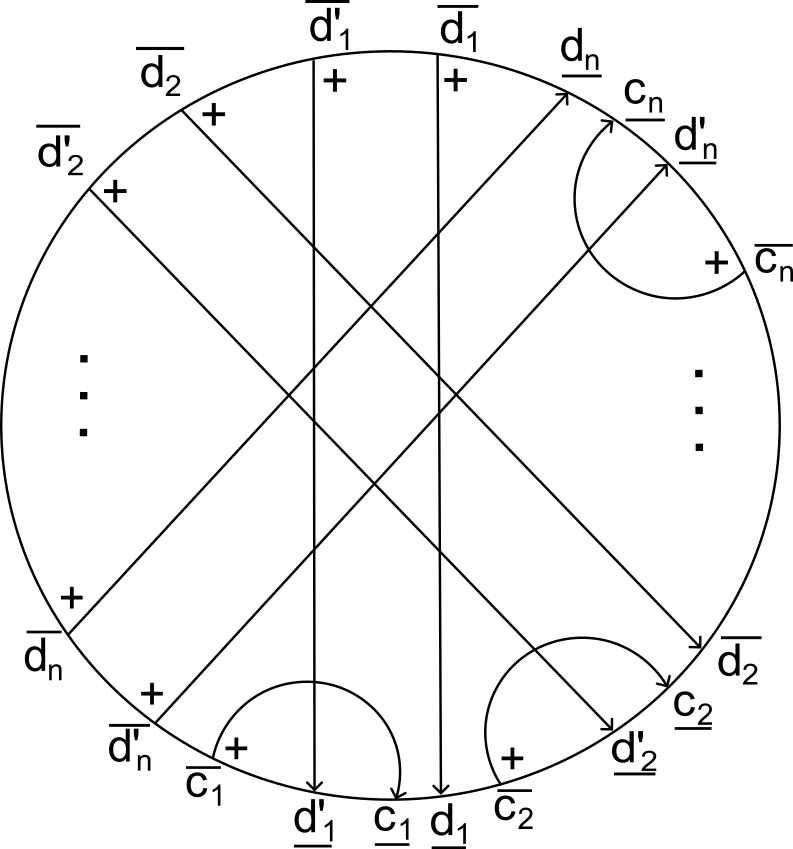}
			\caption{Gauss diagram of $K_n$.}
			\label{fig:gauss}
		\end{figure}

		Now, computing indices for all classical crossings we have,
		\begin{align*}
			ind(d_i)&=(-2n+4i-3) \quad \text{for} \quad i=1,\cdots, n ,\\
			ind(d'_i)&=(-2n+4i-2) \quad \text{for} \quad i=1,\cdots, n ,\\
			ind(c_i)&=1 \quad \text{for} \quad i=1,\cdots, n .
		\end{align*}
		
		Therefore, the odd crossings of $K_n$ are $d_1,\cdots,d_n,c_1,\cdots,c_{n-1}$ and $c_n$. Hence, the odd writhe of $K_n$,
		\[J(K_n)= \sum_{i=1}^n sgn(d_i) + \sum_{i=1}^n sgn(c_i)=2n.\]
		By Lemma~\ref{le:odd}, \[A(K_n) \geq n.\]
		Now, we show that the twisted knot diagram $K_n$ has arc shift number n by applying a sequence of $n$ number of arc shift moves along with generalized Reidemeister moves to obtain a trivial diagram from $K_n$.
		
		Each block $B_i$ transforms into two parallel strands after applying one arc shift move on the red arc $(a,b)$, along with generalized Reidemeister moves (Fig.~\ref{fig:block}). Applying this in every block in $K_n$ gives a trivial diagram without a bar (Fig.~\ref{fig:arcsn}).
	
		\begin{figure}[h]
			\centering 
			\includegraphics[width=0.8\textwidth]{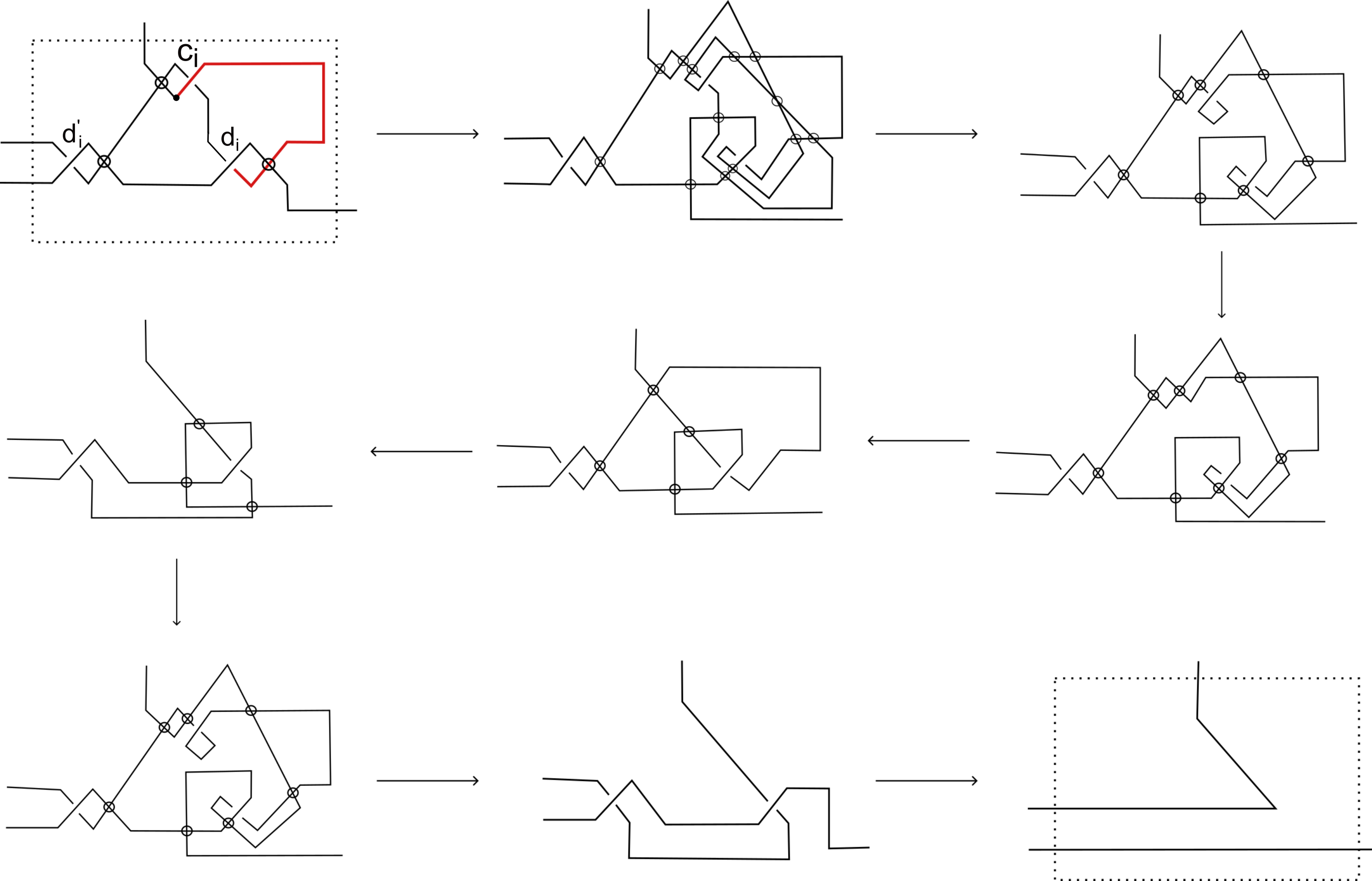}
			\caption{Applying arc shift move in each block.}
			\label{fig:block}
		\end{figure}
		\begin{figure}[h]
			\centering 
			\includegraphics[width=0.7\textwidth]{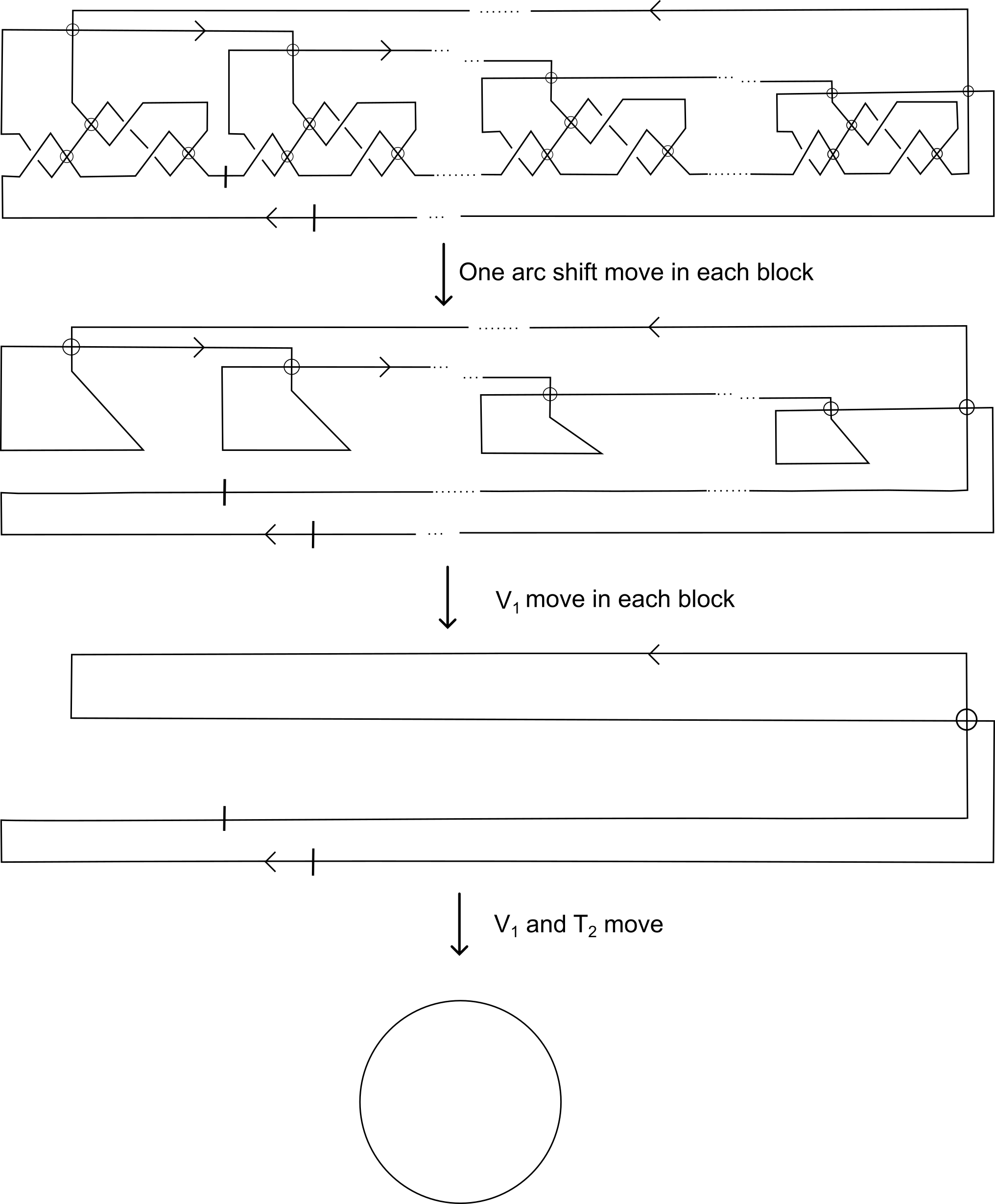}
			\caption{$n$ arc shift moves transforming $K_n$ to a trivial diagram without a bar.}
			\label{fig:arcsn}
		\end{figure}

		Now, to show that each knot in the family is distinct, we compute the polynomial invariant $Q(s,t)$ for $K_n$ $(1 \leq n \leq m)$ \cite{Naoko1}. By Remark~\ref{rem:over}, we need the values of $\bar{\rho}^{O}(c)$, $p^{O}(c)$ and $p^{U}(c)$, which are given in Table~\ref{tab:firstknot}.
		
		\begin{table}
			\centering
			\begin{tabular}{ |c|c|c|c|c|} 
				\hline
			crossings & $ind^{O}(c)$ & $\bar{\rho}^{O}(c)$ & $p^{O}(c)$ & $p^{U}(c)$ \\ 
				\hline 
				$d_1$ & $-2n-2$ & $0$ & $1$ & $1$ \\ 
				\hline
				$d'_1$ & $-2n-1$ & $1$ & $1$ & $1$ \\
				\hline
				$d_i, i=2,\cdots,n$ & $-2n+4i-3$ & $1$ & $0$ & $0$ \\
				\hline
				$d'_i, i=2,\cdots,n$ & $-2n+4i-2$ & $0$ & $0$ & $0$ \\
				\hline
				$c_i, i=1,\cdots,n$ & $1$ & $1$ & $0$ & $0$ \\
				\hline
				
			\end{tabular}
			\caption{}
			\label{tab:firstknot}
		\end{table}
		
		Therefore, for $K_n$$(1 \leq n \leq m)$, \[ Q_{K_n}(s,t) = (st-1)^2+(t-1)^2+(2n-1)(s-1)^2.\]
		Hence, the proof.
	\end{proof}	
	\begin{remark}
		Note that, in Theorem~\ref{th:Kn}, we
		 construct an unknotting sequence $K_n, K_{n-1},\cdots, K_1, K_0$ where $K_0$ is the trivial twisted knot without bar. Also notice that, each $K_i$ has arc shift number $i$ for $1\leq i \leq n$.
	\end{remark}
	In the following theorem, we construct two distinct unknotting sequences of twisted knots and find the arc shift number for each knot in those sequences.
	\begin{theorem}
		For every $m \geq 1$, there exist two distinct unknotting sequences of twisted knots
		 $$ K_m, K_{m-1}, \cdots, K_{1}, K_{0}$$ and 
		 	 $$ K'_m, K'_{m-1}, \cdots, K'_{1}, K'_{0}$$ such that $K_{0}$ and $K'_0$ are trivial twisted knots without and with a bar respectively. Further, the arc shift number for each $K_n$ and $K'_n$ is $n$ for all  $1 \leq n \leq m$.
	\end{theorem}
	\begin{proof}
		For $1 \leq n \leq m$, we take $K_{n}$ and $K'_{n}$ as the twisted knot diagrams shown in Fig.~\ref{fig:twofam}. We denote the classical crossings in $K_{n}$ and $K'_{n}$ by $a_{i}, b_{i}$ and $a'_{i}, b'_{i}, i=1,\dots,n$, respectively. Notice that the only difference between $K_{n}$ and $K'_{n}$ is one bar on the right-most arc of the diagram. 
		\begin{figure}[h!]
			\centering
			\subfloat[$K_n.$]{\includegraphics[width=0.3\textwidth]{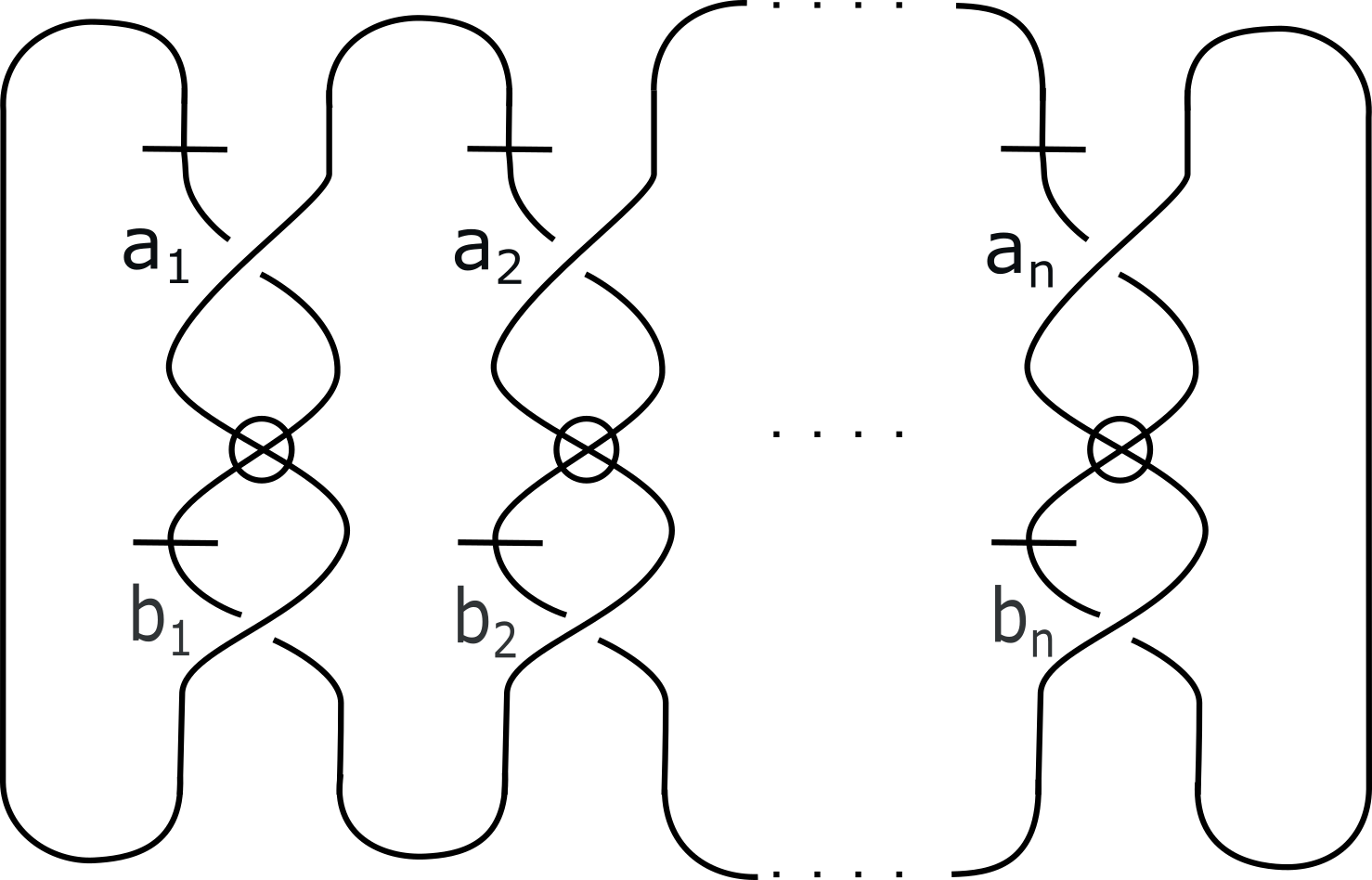}}
			\hspace{1cm}
			\subfloat[$K'_n.$]{\includegraphics[width=0.3\textwidth]{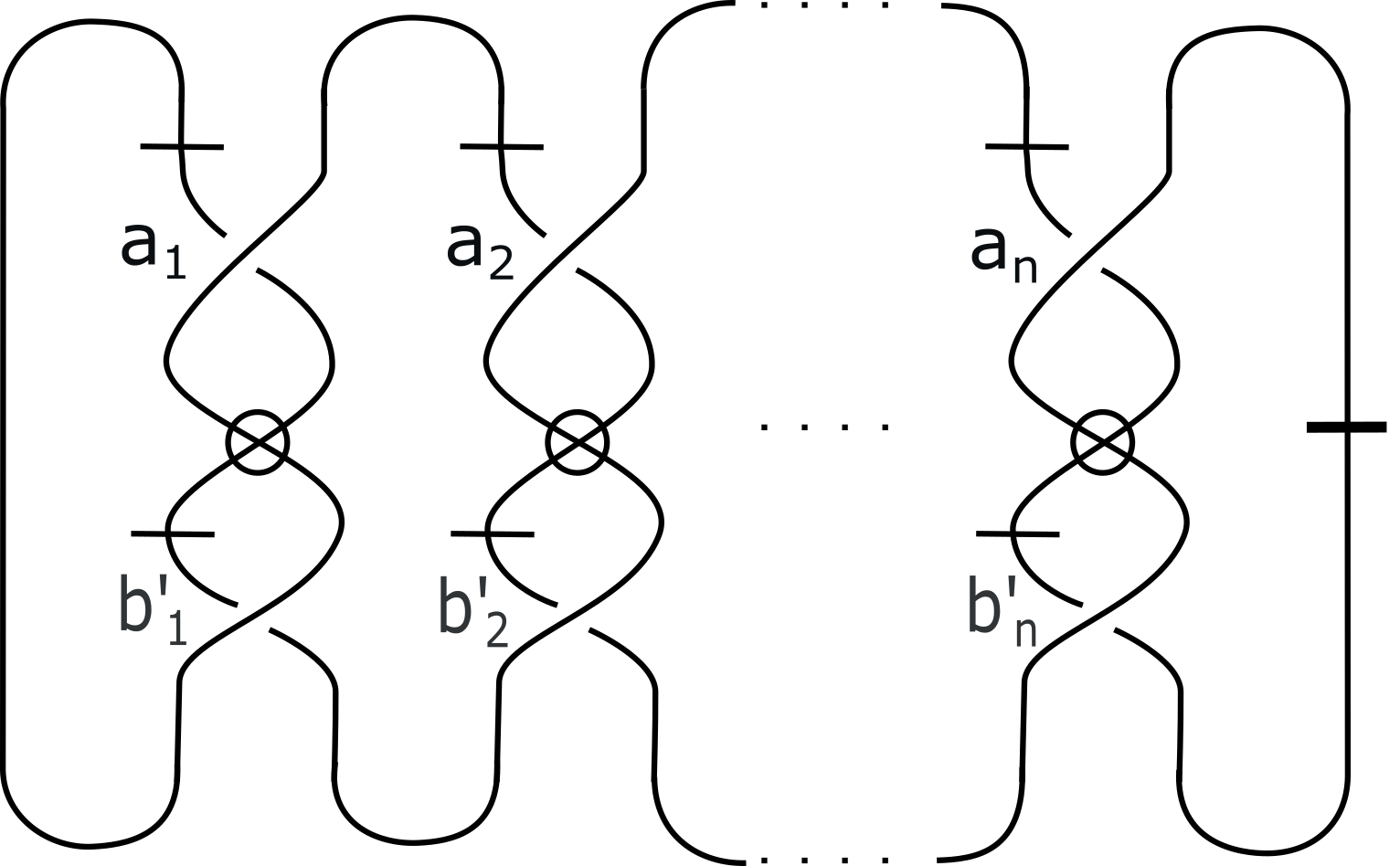}} 
			\caption{}
			\label{fig:twofam}
		\end{figure}
		
		First, we transform the local diagram with one classical crossing and one bar (Fig.~\ref{fig:local1}) to a diagram with one classical crossing, one bar and two virtual crossings. Now, this local transformation converts each vertical block in the diagram of $K_{n}$ and $K'_{n}$ to a single virtual crossing (Fig.~\ref{fig:diagram2}). Therefore, it is easy to observe that the twisted knots $K_{n}$ and $K'_{n}$ transform to a trivial knot without and with a bar, respectively, by applying $n$ number of arc shift moves along with generalized Reidemeister moves (See Fig.~\ref{fig:diagram3},\ref{fig:diagram4}). Therefore,
		\[A(K_{n}) \leq n \quad \text{and} \quad A(K'_{n}) \leq n.\]
		\begin{figure}[h!]
			\centering 
			\includegraphics[scale=0.45]{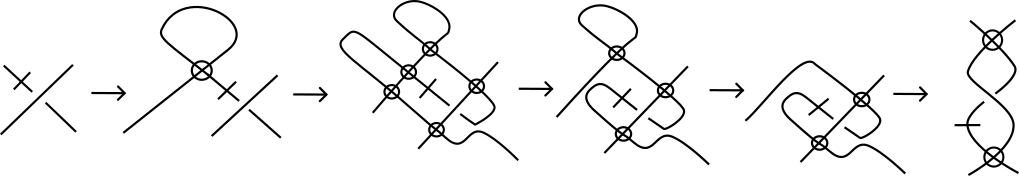}
			\caption{}
			\label{fig:local1}
		\end{figure}
		\begin{figure}[h!]
			\centering 
			\includegraphics[width=0.35\textwidth]{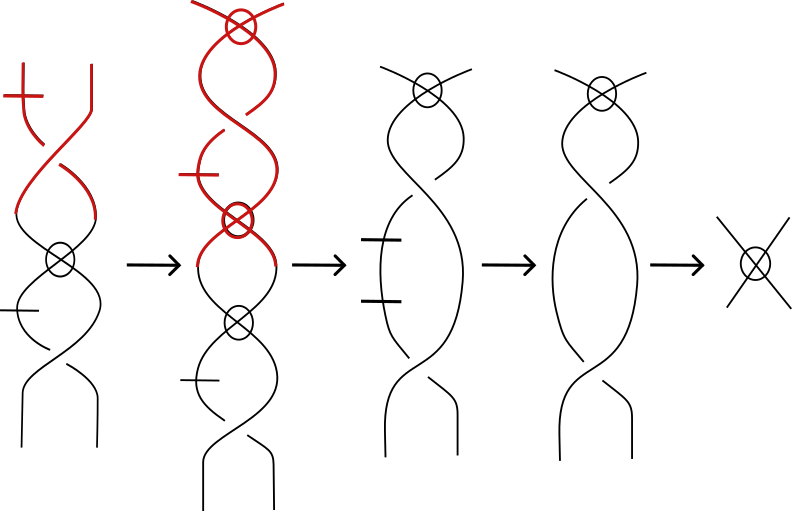}
			\caption{Each vertical block converts to a virtual crossing.}
			\label{fig:diagram2}
		\end{figure}
		\begin{figure}[h!]
			\centering 
			\includegraphics[width=0.7\textwidth]{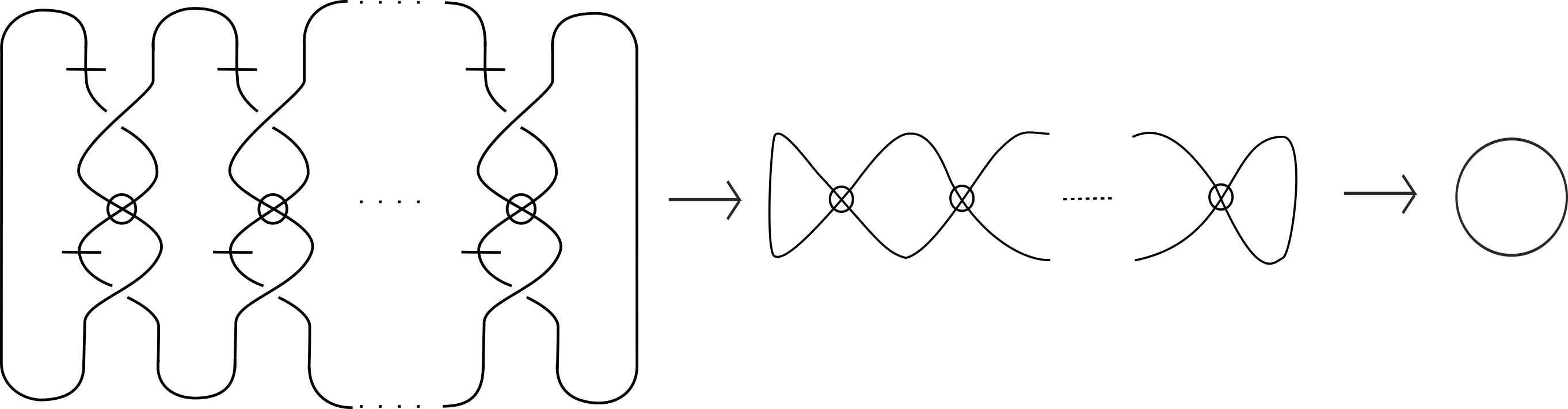}
			\caption{$K_n$ transforms to a trivial diagram without bar.}
			\label{fig:diagram3}
		\end{figure}
		\begin{figure}[h!]
			\centering 

			\includegraphics[width=0.7\textwidth]{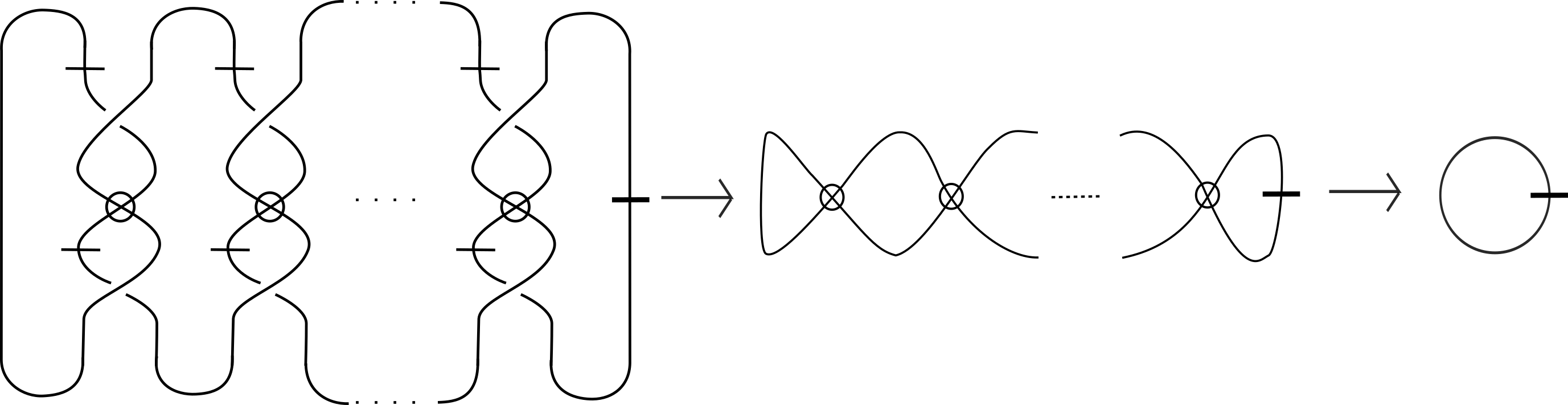}
			\caption{$K'_n$ transforms to a trivial diagram with one bar.}
			\label{fig:diagram4}
		\end{figure}
		Now, the indices for all the classical crossings of $K_{n}$ are as follows.
		\begin{align*}
			& ind(a_{2i}) = 1 =  ind(b_{2i-1}) \; \text{and} \\
			& ind(a_{2i-1}) = -1 =  ind(b_{2i}) .\end{align*} 
		and the indices for all the classical crossings of $K'_n$ are as follows.
		\begin{align*}
			& ind(a'_{2i}) = 1 =  ind(b'_{2i-1}) \quad \text{for} \; i=1,\dots,n, \; \text{and} \\
			& ind(a'_{2i-1}) = -1 =  ind(b'_{2i})  \quad \text{for}\; i=1,\dots,n.
		\end{align*}
		Therefore, the odd writhe for $K_{n}$ and  $K'_{n}$ are, $J(K_{n})= 2n = J(K'_{n})$ and by Lemma~\ref{le:odd}, $$A(K_{n}) \geq n \;  \text{and} \; A(K'_{n})\geq n .$$ \\
		Hence,
		$$A(K'_{n})= n=A( K'_n).$$
			Notice that we obtain $K_{n}$ from $K_{n+1}$ and $K'_{n}$ from $K'_{n+1}$ respectively  for each $n (1\leq n \leq m)$ by applying the local transformation as shown in  Fig.~\ref{fig:diagram2}, in the first vertical block
			 of the diagrams. Hence, we get two unknotting sequences with each twisted knot $K_{n}$ and $K'_{n}$ in the sequences has arc shift number $n$ for all  $1\leq n \leq m$.

		Finally, we show that all the twisted knots in these two sequences are distinct by compute the polynomial invariant $ Q(s,t)$ for $K_{n}$ and  $K'_{n}$ for all $1 \leq n \leq m$. The required values of  $\bar{\rho}^{O}(c)$, $p^{O}(c)$ and $p^{U}(c)$ for $K_{n}$ and  $K'_{n}$ are given in Table~\ref{tab:Kn} and \ref{tab:Kbarn} respectively. Therefore,  
		\[ Q_{K_n}(s,t) = n(st-1)^2+n(s-1)^2 \; \text{for}\;1 \leq n \leq m,\]
		and
		\[ Q_{K'_n}(s,t) = 2n(st-1)(s-1) \; \text{for}\; 1 \leq n \leq m.\]
		Hence the proof.
	\begin{table}[h!]
		\centering
		\begin{tabular}{ |c|c|c|c|c|}
			\hline
			crossings c & $ind^O(c)$ & $\bar{\rho}^{O}(c)$ & $p^{O}(c)$ & $p^{U}(c)$ \\ 
			\hline 
			$a_{2i-1}, i=1,\dots,n$ & $-1$ & $1$ & $1$ & $1$ \\ 
			\hline
			$a_{2i}, i=1,\dots,n$ & $1$ & $1$ & $1$ & $1$ \\
			\hline
			$b_{2i-1}, i=1,\dots,n$ & $1$ & $1$ & $0$ & $0$ \\
			\hline
			$b_{2i}, i=1,\dots,n$ & $-1$ & $1$ & $0$ & $0$  \\
			\hline
			\end{tabular}
		\caption{}
		\label{tab:Kn}
	\end{table}
	\begin{table}[h!]
		\centering
		\begin{tabular}{ |c|c|c|c|c|}
			\hline
			crossings c & $ind^O(c)$ & $\bar{\rho}^{O}(c)$ & $p^{O}(c)$ & $p^{U}(c)$ \\ 
			\hline 
			$a'_{2i-1}, i=1,\dots,n$ & $-1$ & $1$ & $1$ & $0$ \\ 
			\hline
			$a'_{2i}, i=1,\dots,n$ & $1$ & $1$ & $0$ & $1$ \\
			\hline
			$b'_{2i-1}, i=1,\dots,n$ & $1$ & $1$ & $0$ & $1$ \\
			\hline
			$b'_{2i}, i=1,\dots,n$ & $-1$ & $1$ & $1$ & $0$  \\
			\hline
			
		\end{tabular}
		\caption{}
		\label{tab:Kbarn}
	\end{table}
		\end{proof}
	\section{Region Arc Shift Move for Twisted Knots}\label{sec:ras}
	In this section, we generalize the concept of region arc shift move of virtual knot diagrams \cite{Gill1} for the case of twisted knot diagrams. Here, a {\it region} is defined to be a connected component of the 4-valent graph $D_G$ in $\mathbb{R}^2$ obtained from a twisted knot diagram $D$ by putting a vertex at each classical and virtual crossing. Notice that the edges of $D_G$ have bars on it.
	\begin{definition}
		A region arc shift(R.A.S) move at a region $R$ of a twisted knot diagram $D$ is a local transformation that involves applying arc shift moves on each arc of the boundary of $R$. 
	\end{definition}
	\begin{figure}[h]
		\centering 
		\includegraphics[width=0.6\textwidth]{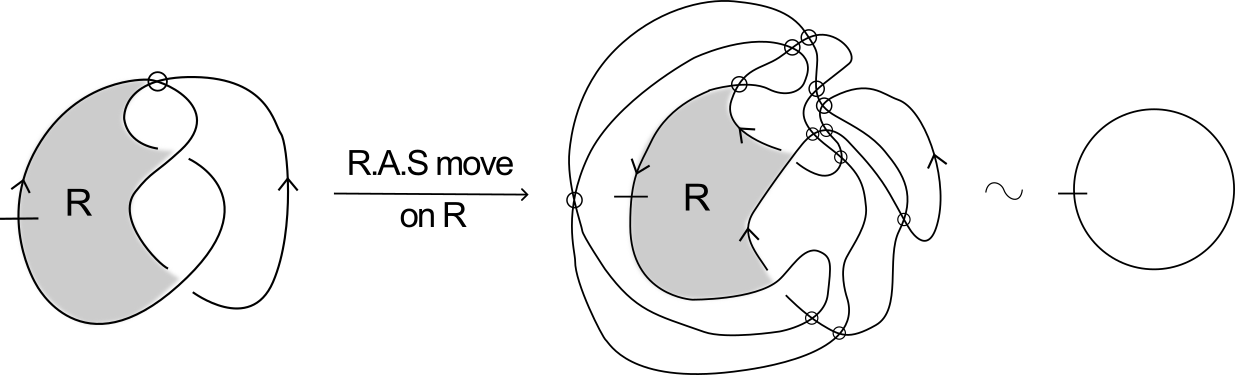}
		\caption{Region arc shift move applied on the region $R$.}
		\label{fig:rarcs}
	\end{figure}	
	In Fig.~\ref{fig:rarcs}, one region arc shift is applied on the region $R$. Applying the region arc shift move on a region of a virtual knot diagram twice yields the original diagram back. This is proved for virtual knots by Gill, Kaur and Prabhakar \cite{Gill1}. A similar argument proves the following proposition.
	\begin{proposition}
		Let $D$ be a twisted knot diagram and $R$ be a region of $D$.
		Applying the region arc shift move twice consecutively on $R$ results in a diagram equivalent to $D$.
	\end{proposition}
	\begin{proof}
		See Proposition~5.1 in \cite{Gill1}.
	\end{proof}
	By Theorem~\ref{th:for}, proving that the region arc shift move is an unknotting operation reduces to showing that the forbidden moves 
	$F_1 (\text{or}\;F_2)$, $F_3(\text{or}\;F_4)$ and $T_4$
	can be realized using region arc shift moves.
	\begin{figure}[h!]
		\centering
		\includegraphics[width=0.5
		\textwidth]{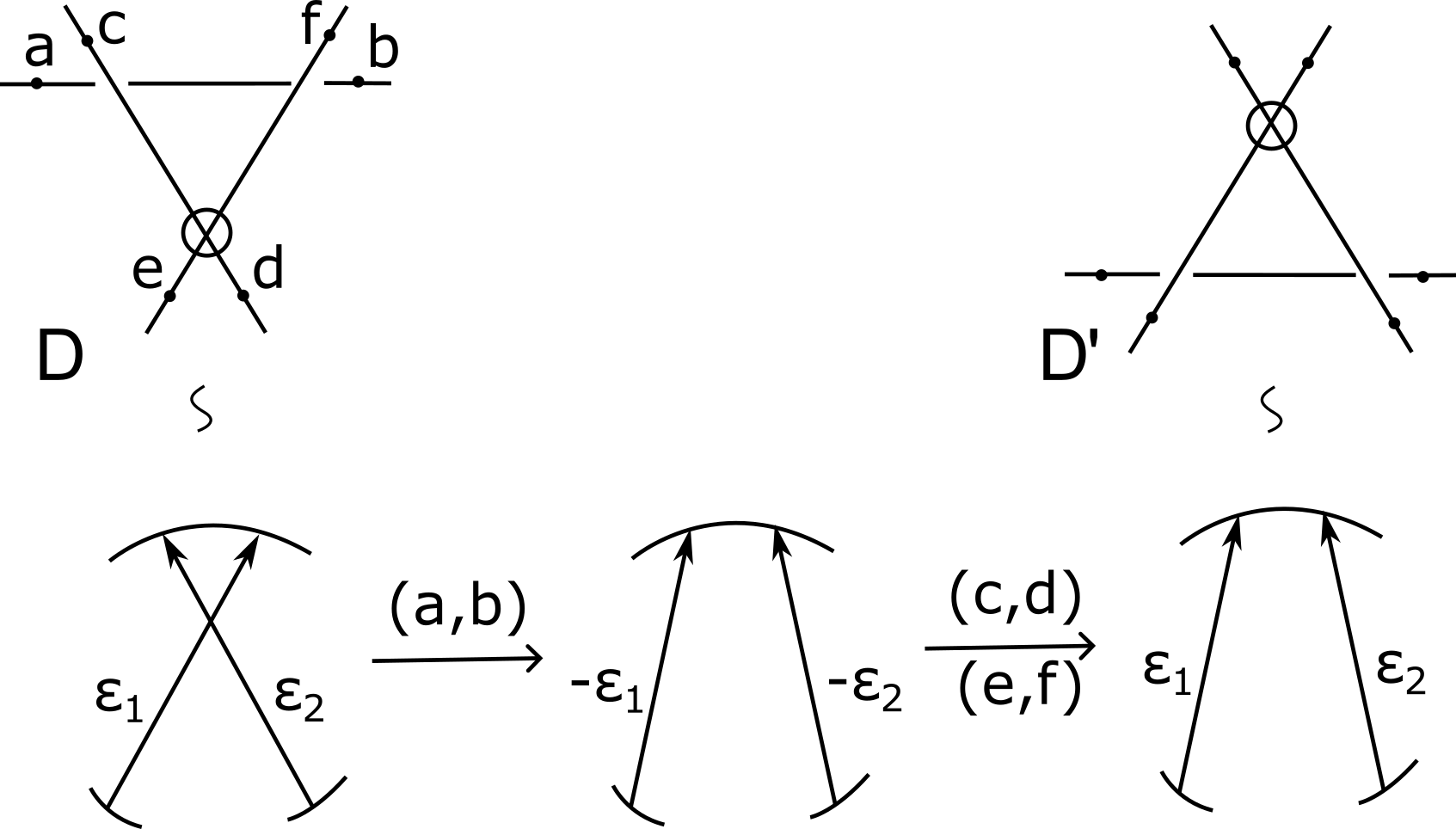}
		\caption{$F_1$ move is realized by one R.A.S.move.}
		\label{fig:forrasF1}
	\end{figure}
	\begin{proposition}\label{pr:F2}
		Let $D$ be a twisted knot diagram and $D'$ be the diagram
		obtained from $D$ by applying forbidden move $F_1 (\text{or}\;F_2)$. Then, there exists a region $R$
		in $D$ such that region arc shift move on region $R$ results in the diagram equivalent to $D'$.
	\end{proposition}
	\begin{proof}
		$F_1$(or $F_2)$ move involves a local diagram with no bars. The same proof for the case of virtual knots is applicable here. See Proposition~5.2 in \cite{Gill1}. See Fig.~\ref{fig:forrasF1}
	\end{proof}
	
	\begin{proposition}\label{pr:F3}
		Let $D$ be a twisted knot diagram and $D'$ be the diagram
		obtained from $D$ by applying forbidden move $F_3 (\text{or}\;F_4)$. Then, there exists a region $R$
		in $D$ such that region arc shift move on region $R$ results in the diagram equivalent to $D'$.
	\end{proposition}
	\begin{proof}
		The forbidden move $F_3$ consists of one region with three boundary arcs: one arc has two classical crossings with a bar, while the other two arcs consist of one classical crossing and one virtual crossing each. Therefore, we need to apply one type 2 arc shift move $\bar{A}'_t$, and two type 1 arc shift moves $A'_{s_{2}}$. From Fig.~\ref{fig:forrasF4}, we observe that the Gauss diagram of the resulting diagram after applying the region arc shift move on $R$ in $D$, is equivalent to the Gauss diagram of $D'$.  A similar
		result can be identically proved for forbidden move $F_4$ also.
	\end{proof}
	\begin{figure}[h]
		\centering
		\includegraphics[width=0.5\textwidth]{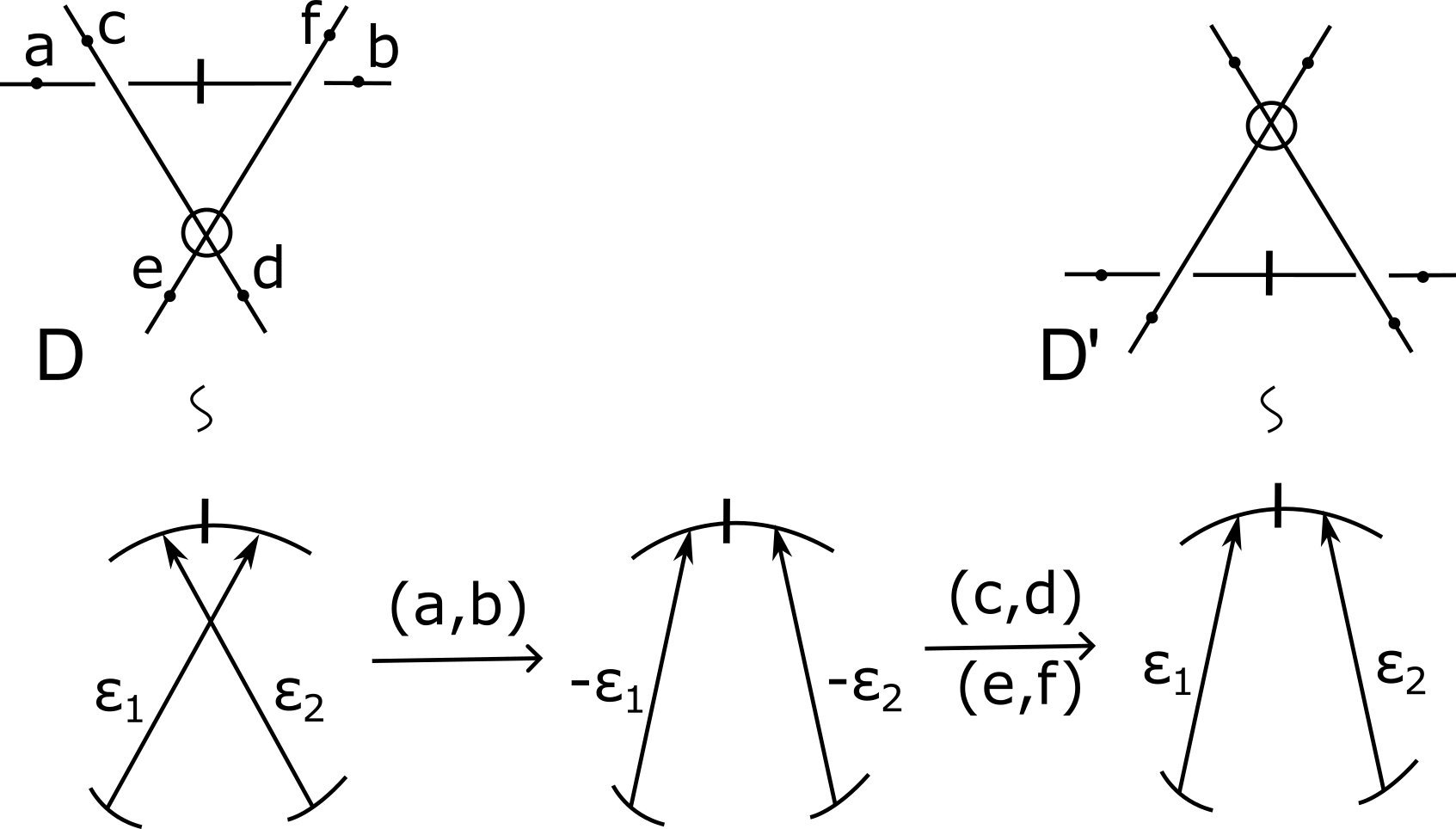}
		\caption{$F_3$ move is realized by one R.A.S.move.}
		\label{fig:forrasF4}
	\end{figure} 
	\begin{proposition}\label{pr:T4}
		Let $D$ be a twisted knot diagram and $D'$ be the diagram
		obtained from $D$ by applying the forbidden move $T_4$. Then, there exists a region $R$
		in $D$ such that region arc shift move on region $R$ results in the diagram equivalent to $D'$.
	\end{proposition}
	\begin{proof}
		The forbidden move $T_4$ involves one region with a boundary arc that has only one classical crossing. As shown in Fig.~\ref{fig:forrasT4}, we introduce a virtual crossing using the $V1$ move and then apply the region arc shift move on the dotted region. Here, we apply one Type 1 arc shift move $A_{s_1}$ and one Type 2 arc shift move $\bar{A}_{s_2}$ on the boundary arcs. We observe from the figure that the resulting diagram after applying the region arc shift move on the dotted region is equivalent to the diagram of $D'$.
	\end{proof}
	\begin{figure}[h]
		\centering
		\includegraphics[width=0.6\textwidth]{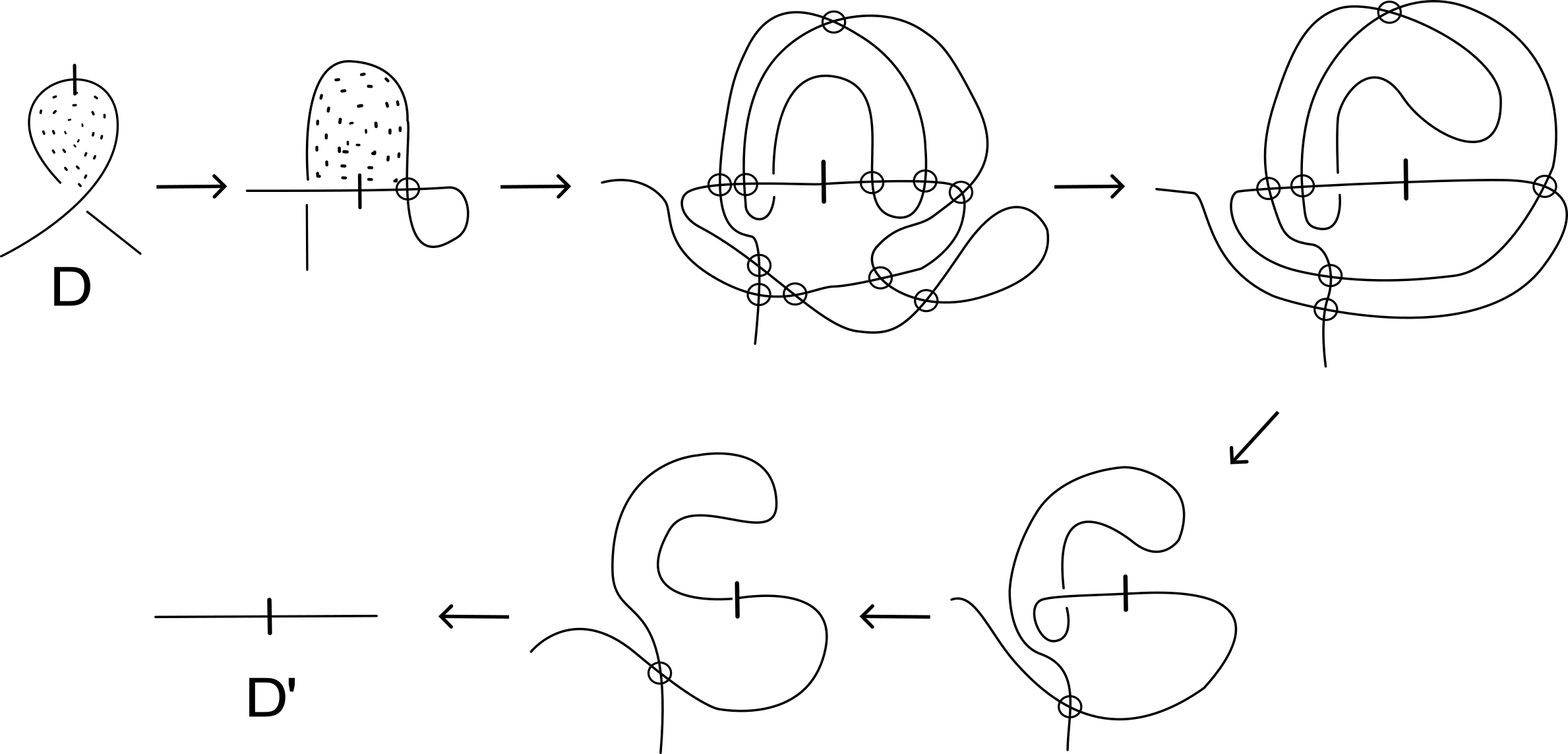}
		\caption{$T_4$ move is realized by one R.A.S move.}
		\label{fig:forrasT4}
	\end{figure}
	From Proposition~\ref{pr:F2},\ref{pr:F3},\ref{pr:T4}, we can conclude the following statement.
	\begin{theorem}\label{th:unop}
		The region arc shift move is an unknotting operation for twisted knots.
	\end{theorem}
	\section{Forbidden number and region arc shift number of twisted knots}\label{sec:rasno}
	In Theorem~\ref{th:unop}, we prove that the region arc shift move is an unknotting operation for twisted knot diagrams. Also, by Theorem~\ref{th:for}, Xue and Deng proved that the set of forbidden moves $\{F_1$(or $F_2$), $F_3$(or $F_2$), $T_4\}$ transforms any twisted knot diagram to a trivial diagram. Motivated by these two results, we define two unknotting numbers, called the Forbidden number, denoted by $\tilde{F}(K)$ and the region arc shift number, denoted by $RA(K)$ of a twisted knot. Moreover, we find a lower bound for $\tilde{F}(K)$ using the polynomial $Q(s,t)$ and an upper bound for $RA(K)$ using its relation with $\tilde{F}(K)$.
	\begin{definition}
		The {\it forbidden number}, denoted by $ \tilde{F}(K)$ of  a twisted knot $K$ is the minimum number of forbidden moves $\{F_1$(or $F_2$), $F_3$(or $F_4$), $T_4\}
		$ needed to transform a diagram of $K$ to a trivial twisted diagram..
	\end{definition}
	\subsection{Forbidden moves and the polynomial $Q(s,t)$}
	In this section, we find a lower bound of the forbidden number using  the polynomial invariant $Q(s,t)$. First, we find the changes in the terms of $Q(s,t)$ after applying a single forbidden move for each of the move $\{F_1$(or $F_2$), $F_3$(or $F_4$), $T_4\}$. 
	\begin{proposition}\label{pr:for1}
		Let $K$ and $K'$ be two twisted knots which can be transformed into each
		other by a single forbidden move $F_1$(or $F_2$). And $\# b(K)$ denotes the number of bars on $K$. Then the following holds.\\
		 $\bullet$ If $\# b(K)$ is even, then
				$$Q_{K'}(s,t)-Q_{K}(s,t)
			=(-1)^{l_{1}} m_{1}(st-1)^{2}+(-1)^{l_{2}} m_{2}(t-1)^{2}+(-1)^{l_{3}}m_{3}(s-1)^{2},
			$$
			where $m_i \in \{0,1,2\}, l_{i}\in \{-1,1\}$ 
			and the coefficients satisfy one of the following mutually exclusive conditions:
			\begin{itemize}
				\item[(i)] $m_1=m_2=m_3=0$,
				\item[(ii)] $m_1=m_2=m_3=1$ and 
				 $(l_1=l_2$ or $l_1=l_3$ or $l_2=l_3)$,
				\item[(iii)] $m_i\in\{0,2\}; i=1,2,3$, with $m_1= m_2 \neq m_3$.
			\end{itemize}
		 $\bullet$ If $\# b(K)$ is odd, then
			$$Q_{K'}(s,t)-Q_{K}(s,t)=
			\pm m_{4}(st-1)(s-1)	$$
			for $m_{4}\in \{0,2\}$.
	\end{proposition}
	\begin{proof}
		We prove this for $F_1$ move. A similar argument applies to the case of $F_2$ move. Consider the Gauss
		diagrams $G$ and $G'$ corresponding to $K$ and $K'$ respectively (Fig.~\ref{fig:f1case}). Let us denote the  chords(and crossings) by $c_1, c_2$ and $c'_1, c'_2$,
		 before and after the $F_1$ move  respectively. 
		
We note that $F_1$ move does not change the sign of the crossings. Then, $sgn(c_1)=\varepsilon_{1}=sgn(c'_1)$ and $sgn(c_2)=\varepsilon_{2}=sgn(c'_2)$. 
		Also, \begin{equation}\label{con:p}
			p^{O}(c_{i})=p^{O}(c'_{i})  \quad \text{and} \quad p^{U}(c_{i})=p^{U}(c'_{i}) \; \text{for} \; i=1,2.
		\end{equation}
		\begin{figure}[h!]
			\centering
			\includegraphics[width=0.45\textwidth]{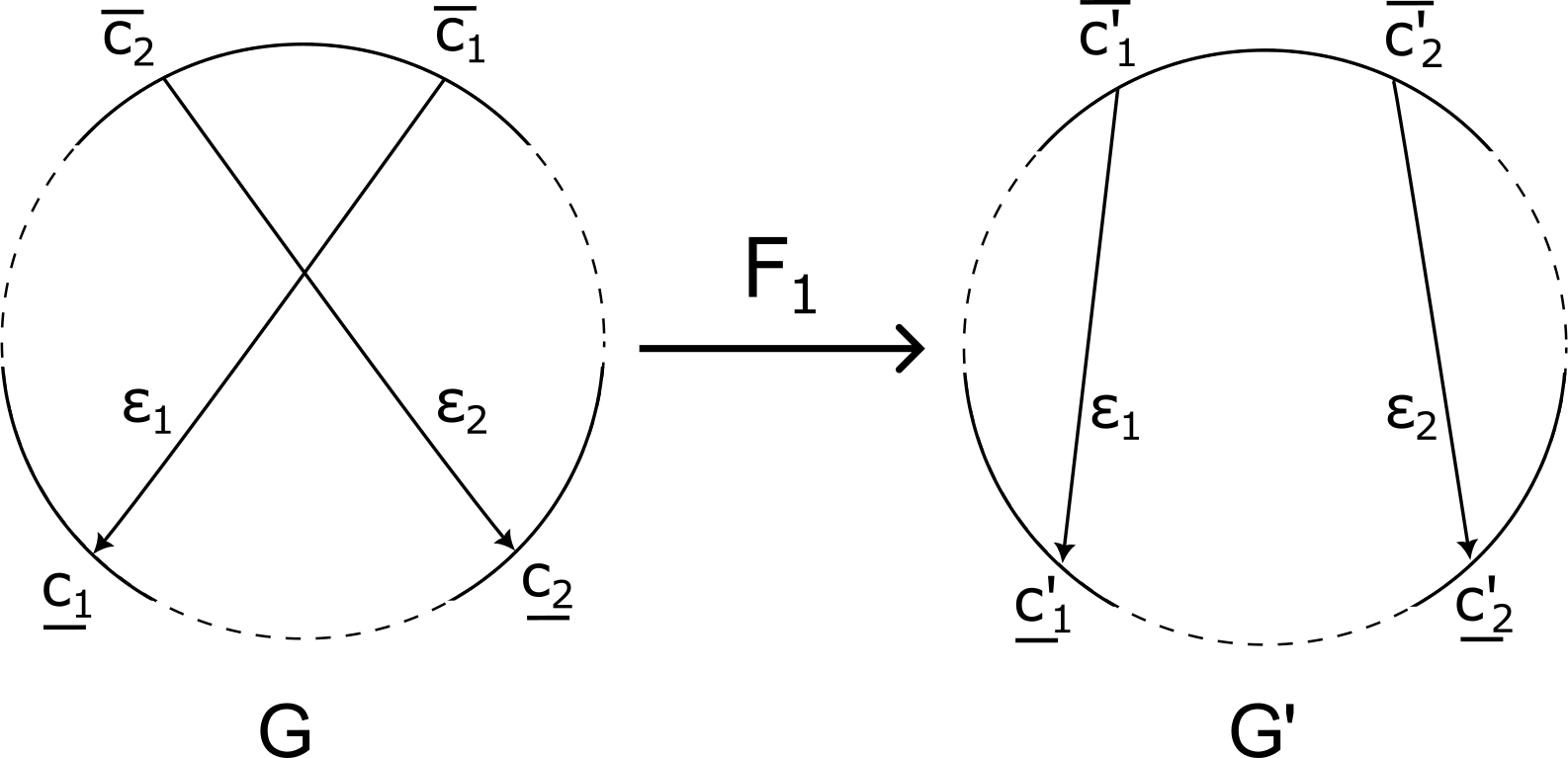}
			\caption{}
			\label{fig:f1case}
		\end{figure}
		Also, from Definition~\ref{def:Q} and  using (\ref{con:p}) we can write,
		\begin{align*}
			Q_{K'}(s,t)-Q_{K}(s,t)
				&= \varepsilon_{1}(s^{\bar{\rho}^{O}(c'_{1})} t^{p^{O}(c_{1})}-1)(s^{\bar{\rho}^{O}(c'_{1})}t^{p^{U}(c_{1})}-1)\\
			&+\varepsilon_{2}(s^{\bar{\rho}^{O}(c_{2})}t^{p^{O}(c_{2})}-1)(s^{\bar{\rho}^{O}(c'_{2})}t^{p^{U}(c_{2})}-1)\\ &-\varepsilon_{1}(s^{\bar{\rho}^{O}(c_{1})}t^{p^{O}(c_{1})}-1)(s^{\bar{\rho}^{O}(c_{1})}t^{p^{U}(c_{1})}-1)\\
			&-\varepsilon_{2}(s^{\bar{\rho}^{O}(c_{2})}t^{p^{O}(c_{2})}-1)(s^{\bar{\rho}^{O}(c_{2})}t^{p^{U}(c_{2})}-1).
		\end{align*}
		From Fig.~\ref{fig:f1case}, we observe that
		\[ind(c'_{1})=ind(c_{1}) + \varepsilon_{2},\]
		\[ind(c'_{2})=ind(c_{1}) - \varepsilon_{1}.\]

		Therefore, the indices are flipped from odd(even) to even(odd) after the move. We discuss
		all cases based on the indices of $c_1,c_2$ of being even or odd.\\
\noindent
		Case I: When both $ind(c_{1})$ and $ind(c_{2})$ are even.
		So, $ind(c'_{1})$ and $ind(c'_{2})$ both are odd and the corresponding values of $\bar{\rho}^{O}$ are as follows.
		\[\bar{\rho}^{O}(c_{1})=0=\bar{\rho}^{O}(c_{2}) \; \text{and} \; \bar{\rho}^{O}(c'_{1})=1=\bar{\rho}^{O}(c'_{2}). \]
		
		Therefore, \begin{align*}
			Q_{K'}(s,t)-Q_{K}(s,t) &=\varepsilon_{1}(st^{p^{O}(c_{1})}-1)(st^{p^{U}(c_{1})}-1)+\varepsilon_{2}(st^{p^{O}(c_{2})}-1)(st^{p^{U}(c_{2})}-1)\\ &-\varepsilon_{1}(t^{p^{O}(c_{1})}-1)(t^{p^{U}(c_{1})}-1)-\varepsilon_{2}(t^{p^{O}(c_{2})}-1)(t^{p^{U}(c_{2})}-1).
		\end{align*}
		When $\# b(K)$ is even, \[p^{O}(c_i)=p^{U}(c_i)\; , i=1,2.\]
		Then for the possible values of $\varepsilon_{1}, \varepsilon_{2}$, $p^{O}(c_i)$ and  $p^{U}(c_i)$ we have,\\
		For $\varepsilon_{1}=\varepsilon=\varepsilon_{2}$, $\varepsilon \in \{1,-1\}$,  \\
		$$Q_{K'}(s,t)-Q_{K}(s,t)
		=\begin{cases}
			2\varepsilon((st-1)^{2}+(t-1)^{2}), \hspace{2cm} \text{for} \; p^{O}(c_i)=1, i=1,2,\\
			\varepsilon((st-1)^{2}-(t-1)^{2}+(s-1)^{2}) , \quad \text{for} \; p^{O}(c_1)=1, p^{O}(c_2)=0,\\
			\varepsilon((st-1)^{2}-(t-1)^{2}+(s-1)^{2}), \quad \text{for} \; p^{O}(c_1)=0, p^{O}(c_2)=1,\\
			2\varepsilon(s-1)^{2}, \hspace{4.15cm} \text{for} \; p^{O}(c_i)=0, i=1,2.
		\end{cases}
		$$
		For $\varepsilon_{1}=\varepsilon=-\varepsilon_{2}$, $\varepsilon \in \{1,-1\}$,  \\
		$$Q_{K'}(s,t)-Q_{K}(s,t)
		=\begin{cases}
			0, \hspace{5.6cm} \text{for} \; p^{O}(c_i)=1, i=1,2\\
			\varepsilon((st-1)^{2}-(t-1)^{2}-(s-1)^{2}) , \quad \text{for} \; p^{O}(c_1)=1, p^{O}(c_2)=0,\\
			\varepsilon(-(st-1)^{2}+(t-1)^{2}+(s-1)^{2}), \; \text{for} \; p^{O}(c_1)=0, p^{O}(c_2)=1,\\
			0, \hspace{5.6cm} \text{for} \; p^{O}(c_i)=0, i=1,2.
		\end{cases}
		$$
		Now, when $\# b(K)$ is odd,
		\[p^{O}(c_i)\neq p^{U}(c_i)\; , i=1,2.\]
		For different values of $\varepsilon_{1}$, $\varepsilon_{2}$, $p^{O}(c_i)$ and  $p^{U}(c_i)$, we  get,
		$$Q_{K'}(s,t)-Q_{K}(s,t)=\begin{cases}
			2\varepsilon(st-1)(s-1),  \quad \text{for} \; \varepsilon_{1}=\varepsilon=\varepsilon_{2}, \varepsilon \in \{1,-1\},\\
			0,  \hspace{3cm} \text{for} \; \varepsilon_{1}=\varepsilon=-\varepsilon_{2},   \varepsilon \in \{1,-1\}.
		\end{cases}	$$
\noindent
	Case II:  One of $ind(c_{1})$ and $ind(c_{2})$ is odd and the other is even. Without loss of generality, let us assume that $ind(c_{1})$ is odd and $ind(c_{2})$ is even.
		So, $ind(c'_{1})$ is even and $ind(c'_{2})$ is odd and the corresponding values of $\bar{\rho}^{O}$ are as follows.
		\[\bar{\rho}^{O}(c_{1})=1=\bar{\rho}^{O}(c'_{2}) \; \text{and} \; \bar{\rho}^{O}(c'_{1})=0=\bar{\rho}^{O}(c_{2}). \]
		Also, \[p^{O}(c_{i})=p^{O}(c'_{i})  \; \text{and} \; p^{U}(c_{i})=p^{U}(c'_{i}) \; \text{for} \; i=1,2.\]
		Therefore, \begin{align*}
			Q_{K'}(s,t)-Q_{K}(s,t) &=\varepsilon_{1}(t^{p^{O}(c_{1})}-1)(t^{p^{U}(c_{1})}-1)+\varepsilon_{2}(st^{p^{O}(c_{2})}-1)(st^{p^{U}(c_{2})}-1)\\ &-\varepsilon_{1}(st^{p^{O}(c_{1})}-1)(st^{p^{U}(c_{1})}-1)-\varepsilon_{2}(t^{p^{O}(c_{2})}-1)(t^{p^{U}(c_{2})}-1).
		\end{align*}
		When $\# b(K)$ is even,\[p^{O}(c_i)=p^{U}(c_i)\; , i=1,2.\]
		Then for all the possible values of  $\varepsilon_{1}$, $\varepsilon_{2}$, $p^{O}(c_i)$ and  $p^{U}(c_i)$ we have the following.\\
		For $\varepsilon_{1}=\varepsilon=\varepsilon_{2}$,  $\varepsilon \in \{1,-1\}$,  \\
		$$Q_{K'}(s,t)-Q_{K}(s,t)
		=\begin{cases}
			0, \hspace{5.6cm} \text{for} \; p^{O}(c_i)=1
			, i=1,2,\\
			\varepsilon((st-1)^{2}+(t-1)^{2}-(s-1)^{2}), \quad \text{for} \; p^{O}(c_1)=1, p^{O}(c_2)=0,\\
			\varepsilon(-(st-1)^{2}-(t-1)^{2}+(s-1)^{2}),\; \text{for} \; p^{O}(c_1)=0, p^{O}(c_2)=1,\\
			0, \hspace{5.65cm} \text{for} \; p^{O}(c_i)=0, i=1,2.
		\end{cases}
		$$
		For $\varepsilon_{1}=\varepsilon=-\varepsilon_{2}$, $\varepsilon \in \{1,-1\}$,  
		$$Q_{K'}(s,t)-Q_{K}(s,t)
		=\begin{cases}
			(-2\varepsilon)((st-1)^{2}-(t-1)^{2}), \hspace{1.65cm} \text{for} \; p^{O}(c_i)=1
			, i=1,2,\\
			\varepsilon(-(st-1)^{2}+(t-1)^{2}-(s-1)^{2}), \quad \text{for} \; p^{O}(c_1)=1, p^{O}(c_2)=0,\\
			\varepsilon(-(st-1)^{2}+(t-1)^{2}-(s-1)^{2}),\quad \text{for} \; p^{O}(c_1)=0, p^{O}(c_2)=1,\\
			(-2\varepsilon)(s-1)^{2}, \hspace{3.9cm} \text{for} \; p^{O}(c_i)=0, i=1,2.
		\end{cases}
		$$
		 When $\# b(K)$ is odd,
		 \[p^{O}(c_i)\neq p^{U}(c_i)\; , i=1,2.\]
		Then for all the possible values of $p^{O}(c_i)$, $p^{U}(c_i)$, $\varepsilon_{1}$ and $\varepsilon_{2}$ we get,
		$$Q_{K'}(s,t)-Q_{K}(s,t)=\begin{cases}
			-2\varepsilon(st-1)(s-1), \quad \text{for} \; \varepsilon_{1}=\varepsilon=-\varepsilon_{2}, \varepsilon \in \{1,-1\},\\
			 0, \hspace{2.9cm}  \quad\text{for} \; \varepsilon_{1}=\varepsilon=\varepsilon_{2}, \varepsilon \in \{1,-1\}.
		\end{cases}	$$
		All the remaining cases follows from the cases discussed above.
	\end{proof}

	\begin{proposition}\label{pr:for3}
		Let $K$ and $K'$ be two twisted knots which transforms into each
		other by a single forbidden move $F_3(\text{or} F_4$).  And $\# b(K)$ denotes the number of bars on $K$. Then the following holds.\\
			$\bullet$ If $\#b(K)$ is even, then
			$$Q_{K'}(s,t)-Q_{K}(s,t)
			=(-1)^{l_{1}} m_{1}(st-1)^{2}+(-1)^{l_{2}} m_{2}(t-1)^{2}+(-1)^{l_{3}}m_{3}(s-1)^{2},
			$$
		where $m_i \in \{0,1,2\}, l_{i}\in \{-1,1\}$ and the
		coefficients satisfy one of the following mutually exclusive conditions:
		\begin{itemize}
			\item[(i)] $m_1=m_2=m_3=0$,
			\item[(ii)] $m_1=m_2=m_3=1$ and 
			( $l_1=l_2$ or $l_1=l_3$).
			\item[(iii)] $m_i\in\{0,2\}$ for all $i$, with $m_1\neq m_2=m_3$ and
			$l_2\neq l_3$.
		\end{itemize}
	$\bullet$ If $\# b(K)$ is odd, then
			$$Q_{K'}(s,t)-Q_{K}(s,t)=
			\pm m_{4}(st-1)(s-1),	$$
			for $m_{4}\in \{0,2\}$.
	\end{proposition}
	\begin{proof}
			We prove this for $F_3$ move. A similar argument applies in the case of $F_4$ move. Consider the Gauss
		diagrams $G$ and $G'$ corresponding to $K$ and $K'$ respectively (Fig.~\ref{fig:f3case}). Let us denote the  chords (and crossings) by $c_1, c_2$ and $c'_1, c'_2$ before and after the $F_3$ move  respectively.
			\begin{figure}[h!]
			\centering
			\includegraphics[width=0.45\textwidth]{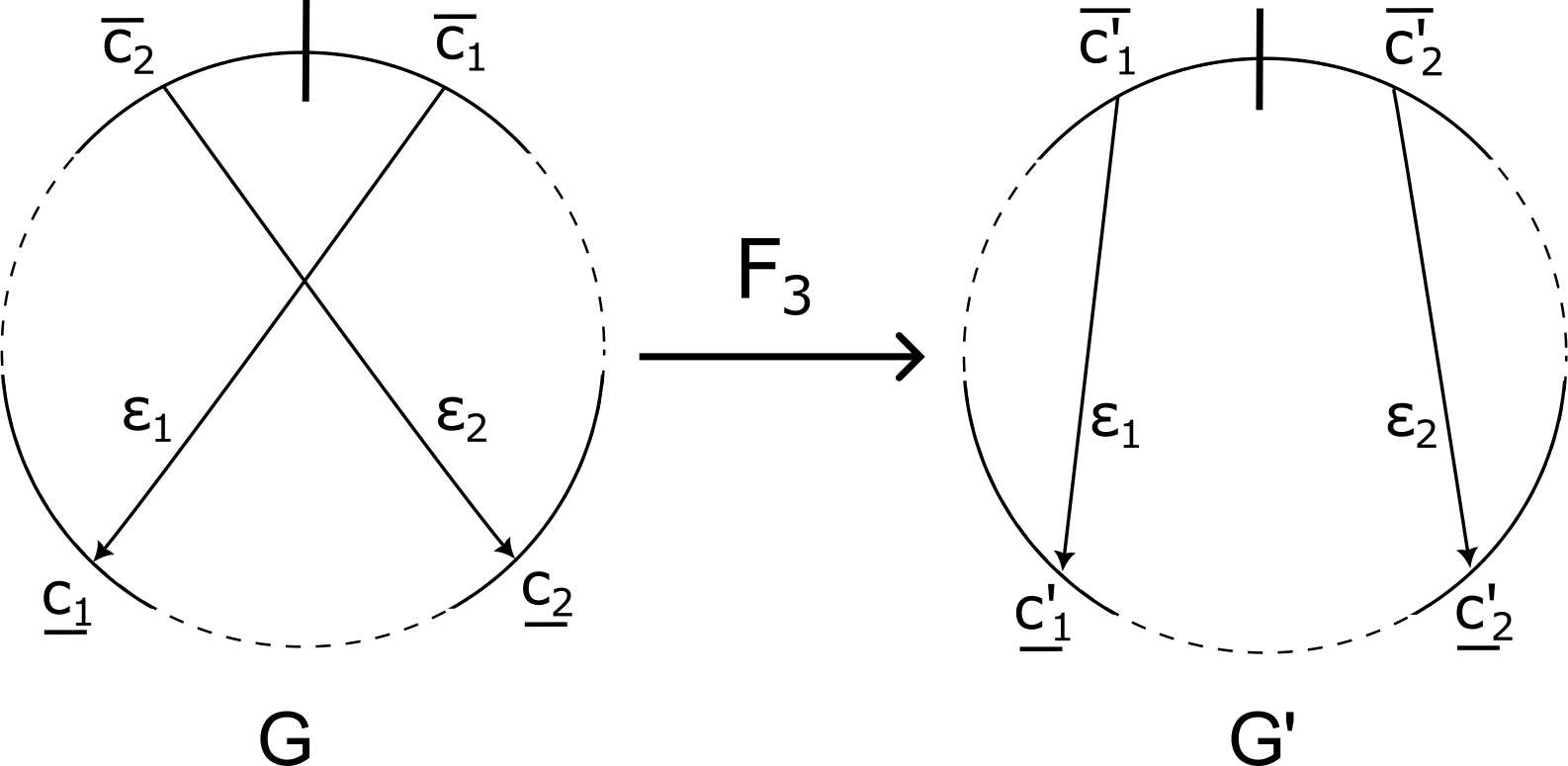}
			\caption{}
			\label{fig:f3case}
		\end{figure}
		
		We note that $F_3$ move does not change the sign of the crossings. Then, $sgn(c_1)=\varepsilon_{1}=sgn(c'_1)$ and $sgn(c_2)=\varepsilon_{2}=sgn(c'_2)$. 
		
		From Fig.~\ref{fig:f3case}, we observe that
		\[ind(c'_{1})=ind(c_{1}) + \varepsilon_{2},\]
		\[ind(c'_{2})=ind(c_{1}) - \varepsilon_{1}.\]
		Since, there is one bar between the chord ends  $\bar{c_1}$ and $\bar{c_2}$,
		\[|p^{*}(c_{i})-p^{*}(c'_{i})|=1 \; \text{for} \; i=1,2 , \; * \in \{O,U\}.\]
		Consequently, the indices change from odd(even) to even(odd) after the move and the values of $p^*, *\in \{O,U\}$ switch between $0$ and $1$. Also,  We discuss
		all the cases based on the indices of $c_1,c_2$ of being even or odd.\\
		\noindent
		Case I: When both $ind(c_{1})$ and $ind(c_{2})$ are even.
		So, $ind(c'_{1})$ and $ind(c'_{2})$ both are odd and the corresponding values of $\bar{\rho}^{O}$ are as follows.
		\[\bar{\rho}^{O}(c_{1})=0=\bar{\rho}^{O}(c_{2}) \; \text{and} \; \bar{\rho}^{O}(c'_{1})=1=\bar{\rho}^{O}(c'_{2}). \]
		
		Therefore, \begin{align*}
			Q_{K'}(s,t)-Q_{K}(s,t) &=\varepsilon_{1}(st^{p^{O}(c'_{1})}-1)(st^{p^{U}(c'_{1})}-1)+\varepsilon_{2}(st^{p^{O}(c'_{2})}-1)(st^{p^{U}(c'_{2})}-1)\\ &-\varepsilon_{1}(t^{p^{O}(c_{1})}-1)(t^{p^{U}(c_{1})}-1)-\varepsilon_{2}(t^{p^{O}(c_{2})}-1)(t^{p^{U}(c_{2})}-1).
		\end{align*}
	 When $\# b(K)$ is even, \[p^{O}(c_i)=p^{U}(c_i)\; , i=1,2.\]
		Then for all the possible values of $\varepsilon_{1}, \varepsilon_{2}$, $p^{O}(c_i)$ and  $p^{U}(c_i)$ we get the following values.\\
		For $\varepsilon_{1}=\varepsilon=\varepsilon_{2}$, where $\varepsilon \in \{1,-1\}$,
		$$Q_{K'}(s,t)-Q_{K}(s,t)
		=\begin{cases}
			2\varepsilon(st-1)^{2}, \hspace{4cm} \text{for} \; p^{O}(c_i)=0, i=1,2,\\
			\varepsilon ((st-1)^{2}+(s-1)^{2}-(t-1)^{2}) , \quad \text{for} \; p^{O}(c_1)=0, p^{O}(c_2)=1,\\
			\varepsilon ((st-1)^{2}+(s-1)^{2}-(t-1)^{2}), \quad \text{for} \; p^{O}(c_1)=1, p^{O}(c_2)=0,\\
			2\varepsilon((s-1)^{2}-(t-1)^{2}), \hspace{2.15cm} \text{for} \; p^{O}(c_i)=1, i=1,2.
		\end{cases}
		$$
		For $\varepsilon_{1}=\varepsilon=-\varepsilon_{2}$, where $\varepsilon \in \{1,-1\}$  \\
		$$Q_{K'}(s,t)-Q_{K}(s,t)
		=\begin{cases}
	\varepsilon ((st-1)^{2}-(s-1)^{2}+(t-1)^{2}) , \quad \text{for} \; p^{O}(c_1)=0, p^{O}(c_2)=1,\\
	\varepsilon (-(st-1)^{2}+(s-1)^{2}-(t-1)^{2}), \; \text{for} \; p^{O}(c_1)=1, p^{O}(c_2)=0,\\
0, \hspace{5.7cm} \text{otherwise}.
		\end{cases}
		$$
	When $\#b(K)$ is odd, we get the following.
		For $p^{O}(c_i),p^{U}(c_i) \in \{0,1\}$,
		$$Q_{K'}(s,t)-Q_{K}(s,t)=\begin{cases}
			2\varepsilon(st-1)(s-1),  \quad \text{for} \; \varepsilon_{1}=\varepsilon=\varepsilon_{2}, \varepsilon \in \{1,-1\},\\
			0,  \hspace{3cm} \text{for} \; \varepsilon_{1}=\varepsilon=-\varepsilon_{2},   \varepsilon \in \{1,-1\}.
		\end{cases}	$$
		\noindent
		Case II:  One of $ind(c_{1})$ and $ind(c_{2})$ is odd and the other is even. Without loss of generality, let us assume that $ind(c_{1})$ is odd and $ind(c_{2})$ is even.
		So, $ind(c'_{1})$ is even and $ind(c'_{2})$ is odd and the corresponding values of $\bar{\rho}^{O}$ are as follows.
		\[\bar{\rho}^{O}(c_{1})=1=\bar{\rho}^{O}(c'_{2}) \; \text{and} \; \bar{\rho}^{O}(c'_{1})=0=\bar{\rho}^{O}(c_{2}). \]
		Therefore, \begin{align*}
			Q_{K'}(s,t)-Q_{K}(s,t) &=\varepsilon_{1}(t^{p^{O}(c'_{1})}-1)(t^{p^{U}(c'_{1})}-1)+\varepsilon_{2}(st^{p^{O}(c'_{2})}-1)(st^{p^{U}(c'_{2})}-1)\\ &-\varepsilon_{1}(st^{p^{O}(c_{1})}-1)(st^{p^{U}(c_{1})}-1)-\varepsilon_{2}(t^{p^{O}(c_{2})}-1)(t^{p^{U}(c_{2})}-1).
		\end{align*}
		When  $\# b(K)$ is even, then
		\[p^{O}(c_i)=p^{U}(c_i)\; , i=1,2.\]
		Then for all the possible values of $\varepsilon_{1}, \varepsilon_{2}$, $p^{O}(c_i)$ and  $p^{U}(c_i)$ we have the following values.\\
		For $\varepsilon_{1}=\varepsilon=\varepsilon_{2}$, where $\varepsilon \in \{1,-1\}$,
		$$Q_{K'}(s,t)-Q_{K}(s,t)
		=\begin{cases}
			\varepsilon (-(st-1)^{2}+(s-1)^{2}-(t-1)^{2}),\hspace{0.15cm} \; \text{for} \; p^{O}(c_i)=1, i=1,2,\\
			\varepsilon ((st-1)^{2}-(s-1)^{2}+(t-1)^{2}), \; \quad\; \text{for} \; p^{O}(c_i)=0, i=1,2,\\
			0, \hspace{5.95cm} \text{otherwise.}
		\end{cases}
		$$
		For $\varepsilon_{1}=\varepsilon=-\varepsilon_{2}$, where $\varepsilon \in \{1,-1\}$,
		$$Q_{K'}(s,t)-Q_{K}(s,t)
		=\begin{cases}
			\varepsilon (-(st-1)^{2}-(s-1)^{2}+
			(t-1)^{2}),  \hspace{0.5cm} \text{for} \; p^{O}(c_i)=1,  i=1,2,\\
			-2\varepsilon(st-1)^{2},  \hspace{3.85cm}\text{for} \; p^{O}(c_1)=1, p^{O}(c_2)=0,\\
			2\varepsilon((t-1)^{2}-(s-1)^{2}),  \hspace{2.4cm} \text{for} \; p^{O}(c_1)=0, p^{O}(c_2)=1,\\
			\varepsilon (-(st-1)^{2}-(s-1)^{2}+(t-1)^{2}), \hspace{0.6cm} \text{for} \; p^{O}(c_i)=0, i=1,2.
		\end{cases}
		$$

		Now, when $\# b(K)$ is odd, then for the possible values of $\varepsilon_{1}, \varepsilon_{2}$, $p^{O}(c_i)$ and  $p^{U}(c_i)$ we have following values.
		$$Q_{K'}(s,t)-Q_{K}(s,t)=\begin{cases}
			-2\varepsilon(st-1)(s-1), \quad \quad \; \text{for} \; \varepsilon_{1}=\varepsilon=-\varepsilon_{2}, \varepsilon \in \{1,-1\},\\
			  0, \hspace{3.1cm}  \quad\text{for} \; \varepsilon_{1}=\varepsilon=\varepsilon_{2}, \varepsilon \in \{1,-1\}.
		\end{cases}	$$
		
		All the remaining cases follows from the cases discussed above.
	\end{proof}
	\begin{proposition}\label{pr:for4}
		Let $K$ and $K'$ be two twisted knots which can be transformed into each
		other by a single forbidden move $T_4$. Then we have
		\[Q_{K'}(s,t)-Q_{K}(s,t) =
		\begin{cases}
			-\varepsilon(t-1)^2, \; \; \text{when}\; \#b(K)\; \text{is even},\\
			0, \hspace{1.6cm} \text{when}\; \#b(K)\;  \text{is odd},
		\end{cases}\] 
		for $\varepsilon \in \{-1,1\}$.
	\end{proposition}
	
	\begin{proof}
		Consider the Gauss
		diagrams $G$ and $G'$ corresponding to $K$ and $K'$. We note that the $T_4$ removes a chord from $G$. Let us denote both the crossing and the chord of $K$ that is involved in the move $T_4$ by $c$. Let us assume that  $sgn(c)=\varepsilon$.
			\begin{figure}[h!]
			\centering
			\includegraphics[width=0.4\textwidth]{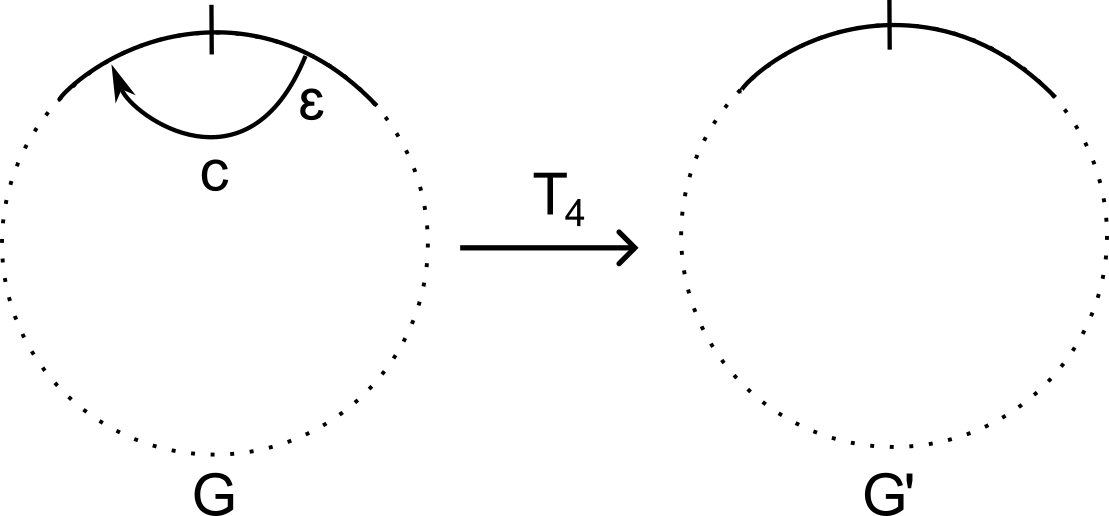}
			\caption{}
			\label{fig:t4case}
		\end{figure}
		From Fig.~\ref{fig:t4case}, we observe that $ind(c)$ is zero i.e even. So, $\bar{\rho}(c)=0$ and $p^{O}(c)=1$.
		Therefore,\[p^{U}(c)=\begin{cases}
			1, \quad \text{when}\; \#b(K) \; \text{is even},\\
			0, \quad \text{when}\; \#b(K) \; \text{is odd}.
		\end{cases}\]
		Hence, for $\varepsilon \in \{-1,1\}$,
		\[Q_{K'}(s,t)-Q_{K}(s,t) =
		\begin{cases}
			-\varepsilon(t-1)^2, \; \text{when}\; \#b(K)\; \text{is even},\\
			0, \hspace{1.6cm} \text{when}\; \#b(K)\;  \text{is odd}.
		\end{cases}\]
		
	\end{proof}
	By Proposition~\ref{pr:for1},\ref{pr:for3},\ref{pr:for4}, we have the following result.
	\begin{theorem}
		Let, $K$ be a twisted knot diagram. Then
		\[\tilde{F}(K) \geq  \frac{J(K)}{4}, \; \text{when} \; b(K) \;\text{is even}.\]
		and
		\[\tilde{F}(K) \geq \frac{J(K)}{2}, \; \text{when} \; b(K) \;\text{is odd
		}.\]
	\end{theorem}
	\begin{proof}
		Let, $\tilde{F}(K)=n$ and $K=K_{n}, K_{n-1}, \dots ,K_{0}$ be the sequence of twisted knots that are obtained while applying $n$ number of the forbidden moves with $K_0$ as the trivial twisted knot and each $K_i$ is obtained from $K_{i+1}$ by one forbidden move along with generalized Reidemeister moves.\\
		Case I: When $\#b(K)$ is even.\\
		The general form of $Q(s,t)$ is given by
		$$Q(s,t)=  a_{n_{1}} (st-1)^2 +  b_{n_{2}} (s-1)^2 +  c_{n_{3}} (t-1)^2,$$
		where $a_{n_{1}}, b_{n_{1}},c_{n_{1}} \in \mathbb{Z} $. Notice that, \[a_{n_{1}} +  b_{n_{2}} = J(K),\] where $J(K)$ is odd writhe, by the definition of $Q(s,t)$.\\
		Now, \begin{equation*}
			|Q_{K}(s,t)-Q_{K_n}(s,t)| \leq |Q_{K}(s,t)-Q_{K_1}(s,t)|+ |Q_{K_1}(s,t)-Q_{K_2}(s,t)|+\cdots+ |Q_{K_{n-1}}(s,t)-Q_{K_n}(s,t)|.
		\end{equation*}
		Using Proposition~~\ref{pr:for1},\ref{pr:for3},\ref{pr:for4} and the fact that $Q_{K_{0}}(s,t)=0$, we can write
		\begin{align*}
			|Q_{K}(s,t)-Q_{K_n}(s,t)| &\leq 2n ( (st-1)^{2}+(t-1)^{2}+(s-1)^{2}),\\
			\text{or,}\; |a_1 (st-1)^2 +  a_2 (s-1)^2 + a_3 (t-1)^2| &\leq 2n ( (st-1)^{2}+(t-1)^{2}+(s-1)^{2}).
		\end{align*}
		Substituting $s=2$ and $t=1$ on both sides, we get,
		$$|a_1 + a_2 |\leq 4n,$$
		which implies
		\[\tilde{F}(K) \geq \frac{|J(K)|}{4}.\]
		\noindent
		Case II: When $\#b(K)$ is odd.\\
		The general form of $Q(s,t)$ is given by
		$$Q(s,t)=  a_4 (st-1)(s-1), $$
		where $a_4 \in \mathbb{Z} $. Notice that, \[a_4 = J(K),\] where $J(K)$ is odd writhe, by the definition of $Q(s,t)$.
		Using Proposition~~\ref{pr:for1},\ref{pr:for3},\ref{pr:for4},  and the fact that $Q_{K_{n}}(s,t)=0$, we can write
		\begin{align*}
			|Q_{K}(s,t)-Q_{K_n}(s,t)| &\leq 2n (st-1)(s-1),\\
			\text{or,}\; |a_4 (st-1)(s-1)| &\leq 2n (st-1)(s-1),
			\text{or,}\; n \geq \frac{|a_4|}{2}.
		\end{align*}
		Hence,
		\[\tilde{F}(K) \geq \frac{|J(K)|}{2}.\]
		
	\end{proof}
	Now, we define  the region arc shift number as follows.
	\begin{definition}
		The {\it region arc shift number}, denoted by $ RA(K)$ of  a twisted knot $K$ is the minimum number of region arc shift moves needed to transform a diagram of $K$ to a trivial twisted diagram.
		
	\end{definition}
	By definition, we can see that $ RA(K)$ is a  twisted knot invariant. Also, $ RA(K)=0$ implies the twisted knot $K$ is a trivial twisted knot.
	\begin{example}
		The twisted knot $K$ in Fig.~\ref{fig:rarcs} has $RA(K)=1$.
	\end{example} 
	\begin{theorem}
		Let $\tilde{F}(K)$ be the forbidden number of a twisted knot $K$. Then, $RA(K) \leq \tilde{F}(K)$.
	\end{theorem}
	\begin{proof}
		Let the forbidden number of a twisted knot, $\tilde{F}(K)=n$. Therefore, there exists a diagram $D$ of $K$ such that applying $n$ number of forbidden moves $(F_1$(or $F_2$), $F_3$(or $F_4$), $T_4)$ on $D$ results into a trivial twisted knot diagram. Now, by Proposition~\ref{pr:F2},\ref{pr:F3},\ref{pr:T4}, each forbidden move can be realized by one region arc shift move. Therefore, applying $n$ region arc shift move transforms $D$ into a trivial twisted diagram. Hence,
		\[RA(K) \leq n=\tilde{F}(K)\]
	\end{proof}
	\begin{example}
		In the following example (Fig.~\ref{fig:regionex2}), the knot $K$ has $RA(K)=1$. Applying the region arc shift move on the dotted region results into a trivial diagram with one bar. Also, 
		$\tilde{F}(K)= 2$. This is an example where the strict  inequality $RA(K) < \tilde{F}(K)$ holds.
		\begin{figure}[h!]
			\centering 
			\includegraphics[width=0.15
			\textwidth
			]{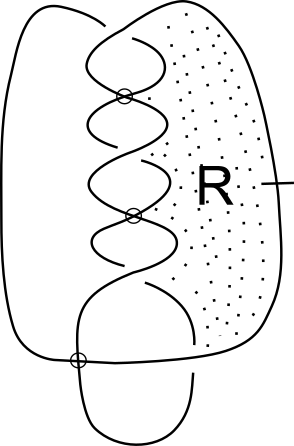}
			\caption{}
			\label{fig:regionex2}
		\end{figure}	
	\end{example}

	\begin{theorem}
		For each positive integer \(n \geq 1\), there exists a twisted knot \( K_n \) such that the region arc shift number of \(K_n\) is less than or equal to \(n\).
	\end{theorem}
	\begin{proof}
		We take $K_{n}$ as the twisted knot diagram shown in Fig.~\ref{fig:rasfam1}. 
		We denote the classical crossings in $K_{n}$ by $c_{i}$ and $c'_{i}, i=1,\dots,n$.
		
		\begin{figure}[h!]
			\centering 
			\includegraphics[width=0.45\textwidth,height=0.25\textwidth]{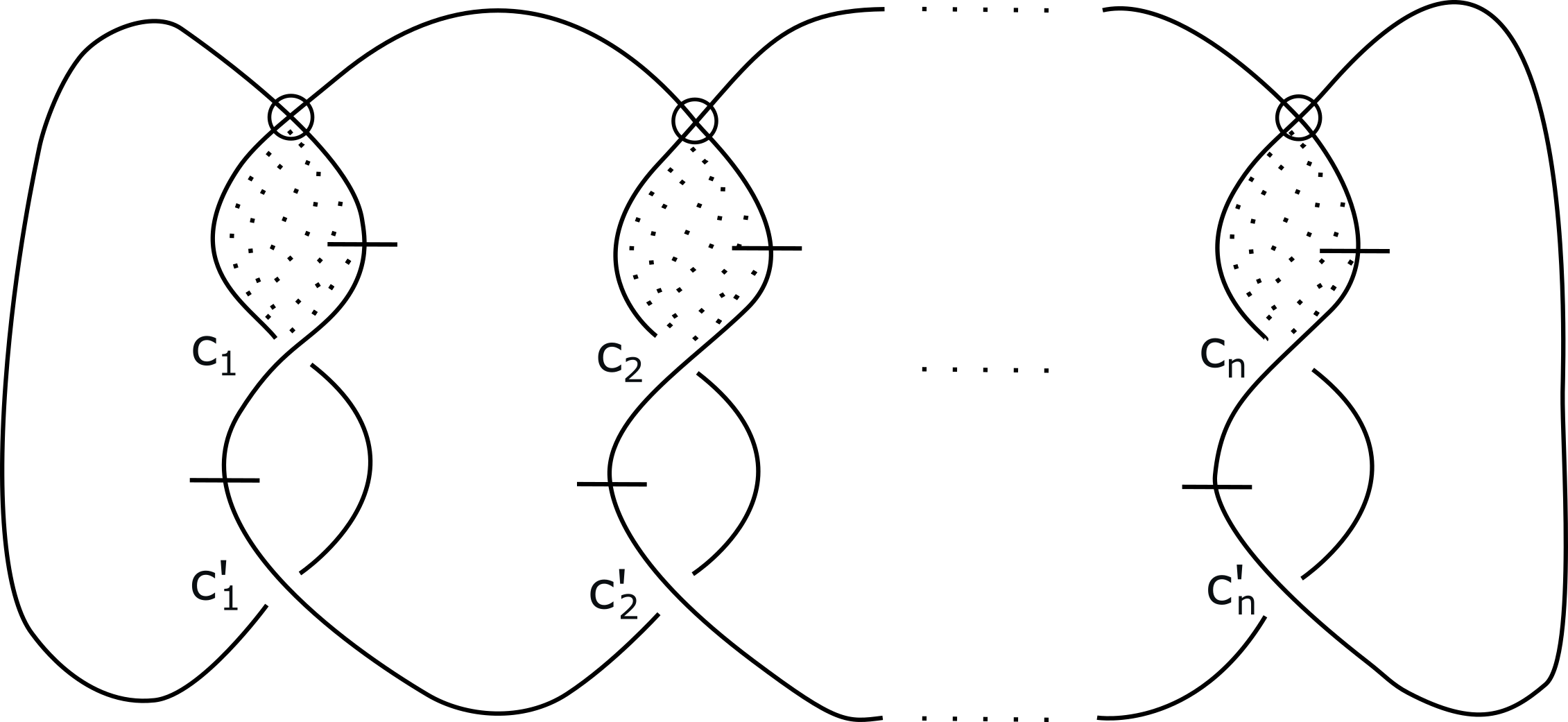}
			\caption{$K_{n}.$}
			\label{fig:rasfam1}
		\end{figure}
			\begin{figure}[h!]
			\centering 
			\includegraphics[width=0.7\textwidth,height=0.3\textwidth]{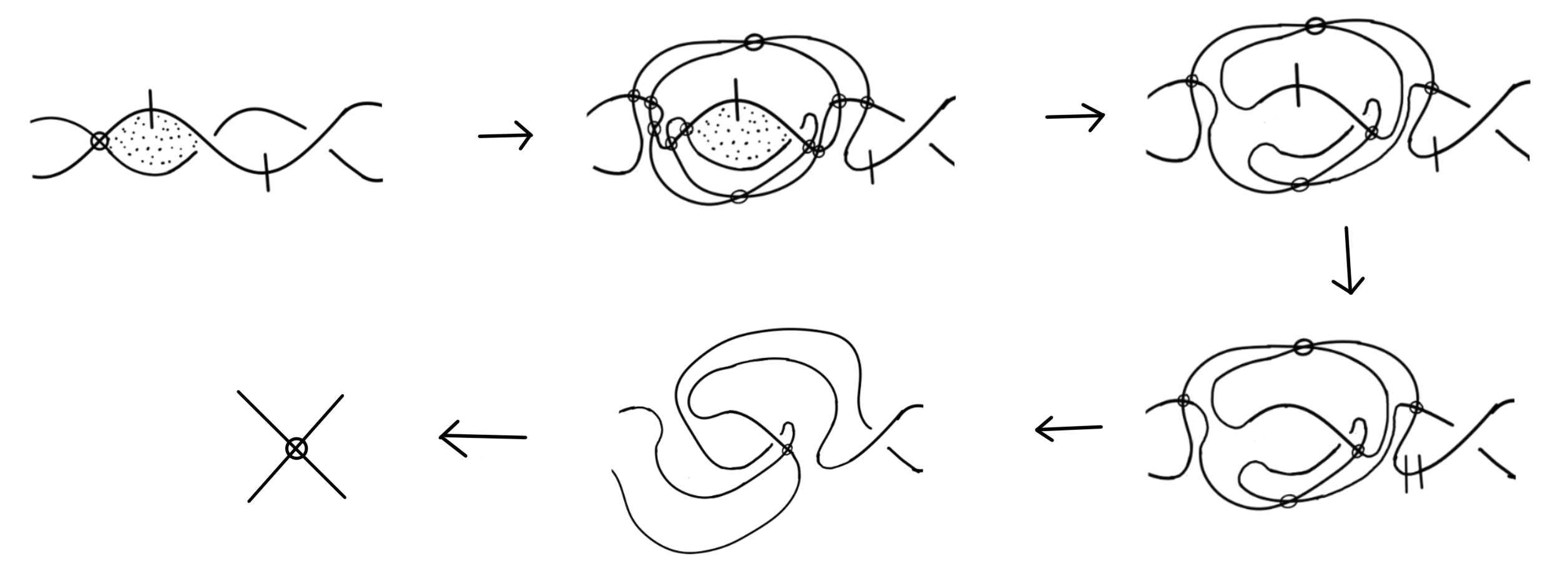}
			\caption{Local diagram transforms into a virtual crossing after one R.A.S moves.}
			\label{fig:localras}
		\end{figure} 
		First, we show that the local diagram in Fig.~\ref{fig:localras} converts into a virtual crossing after applying a R.A.S move in the dotted region. Therefore, it is easy to observe that the twisted knot $K_{n}$ can be transformed to a trivial twisted knot by applying $n$ R.A.S moves and generalized Reidemeister moves (See Fig.~\ref{fig:rasfam2}). Therefore,
		\[A(K_{n}) \leq n.\]
		\begin{figure}[h!]
			\centering 
			\includegraphics[width=0.5
			\textwidth,height=0.5\textwidth]{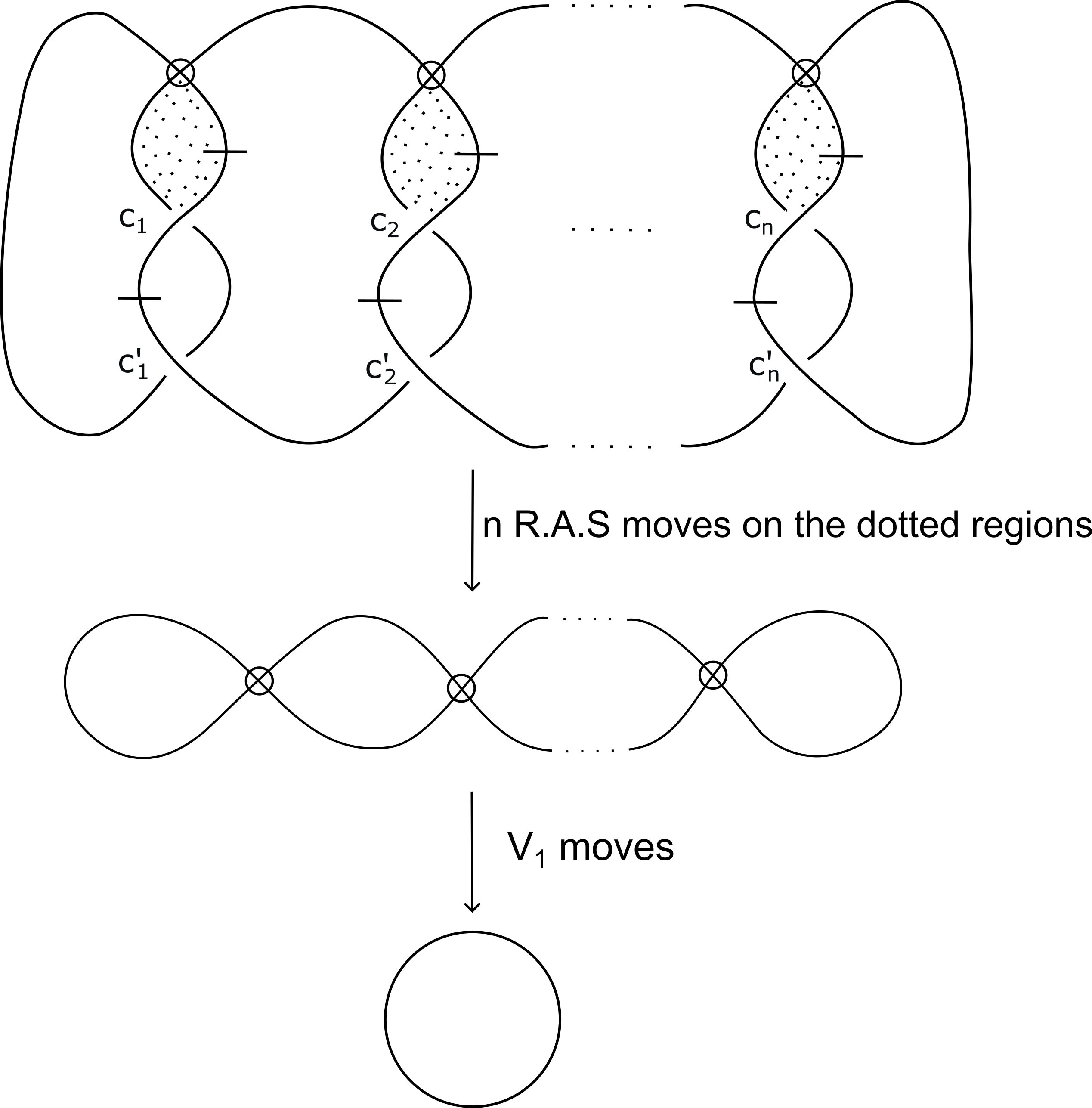}
			\caption{$K_{n}$ transforms to a trivial twisted knot after $n$ R.A.S moves.}
			\label{fig:rasfam2}
		\end{figure}

		Now, we compute the polynomial invariant $ Q(s,t)$ for $K_{n}$ to show that $K_n$ are distinct for every $n\geq 1$. The required values of  $\bar{\rho}^{O}(c)$, $p^{O}(c)$ and $p^{U}(c)$ for $K_{n}$ are given in Table~\ref{tab:$Kn$}. 
		
		\begin{table}[h!]
			\centering
			\begin{tabular}{ |c|c|c|c|c|}
				\hline
				crossings c & $ind^O(c)$ & $\bar{\rho}^{O}(c)$ & $p^{O}(c)$ & $p^{U}(c)$ \\ 
				\hline 
				$c_{(2i-1)}$ & $-1$ & $1$ & $1$ & $1$ \\ 
				\hline
				$c_{(2i)}$ & $1$ & $1$ & $1$ & $1$ \\
				\hline
				$c'_{(2i-1)}$ & $-1$ & $1$ & $0$ & $0$ \\
				\hline
				$c^{\prime}_{(2i)}$ & $1$ & $1$ & $0$ & $0$  \\
				
				\hline
				
			\end{tabular}
			\caption{}
			\label{tab:$Kn$}
		\end{table}
		Therefore,  \[ Q_{K_n}(s,t) = n(st-1)^2-n(s-1)^2 \; \text{for}\; n \geq 1.\]
	\end{proof}

	\section*{Acknowledgments}
	The authors acknowledge the financial support from the Anusandhan National Research Foundation (ANRF), Government of India(Grant no: CRG/2023/004921/343).

	\bibliographystyle{plain}

\end{document}